\documentclass{amsart}
\usepackage{color}
\usepackage{amsmath,amssymb,amsthm}
\usepackage[margin=1in]{geometry}
\usepackage{amssymb,txfonts,xspace}
\usepackage{graphicx}
\usepackage{mathtools}
\usepackage[normalem]{ulem}
\usepackage{tikz}
 \usetikzlibrary{arrows}
 \usetikzlibrary{decorations.pathmorphing}
 \usetikzlibrary{shapes}
 \usetikzlibrary{decorations.markings}

\newcommand{\Z}{\mathbb Z} 
\newcommand{\R}{\mathbb R} 
\newcommand{\B}{{\mathbb B}^{d_Q}} 
\newcommand{\D}{\mathbb D} 
\newcommand{\Neigh}[1]{\mathcal{N}^{d_Q}_{\varepsilon}(#1) } 
\newcommand{\ANeigh}[1]{\mathcal{N}^{d_{Q^*}}_{\varepsilon}(#1) }

\DeclareMathOperator{\Sad}{Sa}  
\DeclareMathOperator{\Sa}{Sa_Q}  
\DeclareMathOperator{\ASa}{S\acute{a}_Q} 
\DeclareMathOperator{\ASad}{S\acute{a}} 
\DeclareMathOperator{\Sar}{Sa_{rQ}}  

\DeclareMathOperator{\Iso}{C}  
\DeclareMathOperator{\AIso}{C'} 

\DeclareMathOperator{\Bubb}{Bubb} 
\DeclareMathOperator{\supp}{Supp} 

\newcommand{\Fx}{\frac{\partial F}{\partial x}}
\newcommand{\Fy}{\frac{\partial F}{\partial y}}
\newcommand{\Fz}{\frac{\partial F}{\partial z}}

\newcommand{\fx}{\frac{\partial f}{\partial x}}
\newcommand{\fy}{\frac{\partial f}{\partial y}}

\newcommand{\falpha}{\frac{\partial f}{\partial \alpha}}
\newcommand{\fbeta}{\frac{\partial f}{\partial \beta}}

\newcommand{\fxx}{\frac{\partial^2 f}{\partial x^2}}
\newcommand{\fxy}{\frac{\partial^2 f}{\partial x \partial y}}
\newcommand{\fyx}{\frac{\partial^2 f}{\partial y \partial x}}
\newcommand{\fyy}{\frac{\partial^2 f}{\partial y^2}}

\newcommand{\phix}{\frac{\partial \varphi}{\partial x}}
\newcommand{\phiy}{\frac{\partial \varphi}{\partial y}}
\newcommand{\Abs}[1]{\left| #1 \right|  } 
\DeclareMathOperator{\Sgn}{Sgn} 
\DeclareMathOperator{\M}{\mathcal{M}} 
\DeclareMathOperator{\Div}{div} 

\newtheorem{theorem}{Theorem}
\newtheorem{lemma}[theorem]{Lemma}
\newtheorem{proposition}[theorem]{Proposition}

\newtheorem{corollary}[theorem]{Corollary}
\newtheorem{conjecture}[theorem]{Conjecture}

\newtheorem{remark}[theorem]{Remark}
\newtheorem{example}[theorem]{Example}
\newtheorem{definition}[theorem]{Definition}

\begin{document}

\title{Sub-Finsler Heisenberg Perimeter Measures}
\author{Ayla P. S\'anchez}

\begin{abstract}
This work is an investigation of perimeter measures in the metric measure space given by 
the Heisenberg group with Haar measure and a Carnot-Carath\'eodory metric, which is in general a sub-Finsler metric. 
Included is a reduction of Minkowski content in {\em any} CC-metric to an integral formula in terms of Lebesgue surface area in $\R^3$.
Using this result, I study two perimeter measures that arise from the study of Finsler normed planes, 
and provide evidence that Pansu's conjecture regarding the isoperimetric problem in the sub-Riemannian case 
appears to hold in the more general sub-Finsler case.  
This is contrary to the relationship between Finsler and Riemannian isoperimetrices.
In particular, I show that for any CC-metric there exist a class of surfaces with Legendrian foliation by CC-geodesics generalizing Pansu's bubble set,
but that even in their natural metric using either perimeter measure, the computed examples of such surfaces have lower isoperimetric ratio 
than the Pansu bubble set, which has Legendrian foliation by sub-Riemannian geodesics.
\end{abstract}

\maketitle

\section{Introduction} \label{sec-intro}
\subsection{The isoperimetric problem}
In this paper I study the Heisenberg group, 
which I will consider under exponential coordinates $H(\R) \cong (\R^3, *)$ , where 
$(x_1, y_1, z_1)*(x_2, y_2, z_2) = (x_1 + x_2, y_1,+ y_2, z_1 + z_2 - \frac{1}{2}(x_1 y_2 - x_2 y_1)$. 
It has a natural notion of volume as Haar measure, which can be taken to be Lebesgue measure $\lambda$.
I equip $H(\R)$ with a Carnot-Carath\'eodory metric $d_Q$ that arises from a norm in the plane $||\cdot||_Q$ 
and is in general sub-Finsler.
The choice of the Euclidean norm yields the sub-Riemannian case, in which there are multiple equivalent definitions
for surface area.  In this case, an open conjecture by Pansu claims that isoperimetrices are so-called {\em bubble sets},
which are the only sets with Legendrian foliation by geodesics (which implies constant mean curvature) \cite{Pansu2,Pansu1}.
Partial results show Pansu's conjecture holds for $C^2$ smooth sets \cite{RR2}, radially symmetric sets \cite{Monti1}, and convex sets \cite{MR}.

Much less studied are the more general sub-Finsler cases.  
Difficulty arises immediately as there are multiple perimeter measures one could use to define surface area.
Two natural choices come from the study of norms in the plane, and can both described in terms of Minkowski content.
They are, in general, unwieldy and difficult to compute.  This paper contains a study of such sub-Finsler perimeter measures and 
the generalized isoperimetric problem, which asks : ``For the metric measure space $(H(\R), d_Q, \lambda)$ 
equipped with a given perimeter measure $\mu$, does there exist a
maximal isoperimetric ratio $\lambda^{3/4} / \mu$ and are there any sets that achieve it?''

\subsection{Main results}
The first main result reducing Minkowski content to an integral formula is found in Section \ref{subsec:thm1}.

\begin{theorem}[Sub-Finsler Minkowski content for implicit surfaces]\label{cor:MinkFormula} 
Let $(H(\R), d_Q, \lambda)$ be the sub-Finsler metric measure space given by
the Heisenberg group $H(\R)$  equipped with $d_Q$, a CC-metric arising from the norm $||\cdot||_Q$ 
where $Q$ is a closed, centrally symmetric, convex plane figure $Q$,
and $\lambda$ the 3-dimensional Lebesgue measure.
Then the Minkowski content of a $C^1$ smooth regular surface $S \subset (H(\R), d_Q, \lambda)$ parametrized by $F(x,y,z) = 0$
in exponential coordinates  can be computed in given in terms of the Lebesgue surface area measure $\sigma$ as follows: 
$$  \lim_{\varepsilon\to0} \frac{ \lambda(\Neigh{S}) - \lambda(S)}{\varepsilon} =
\iint_{S}\left| \left|  \, \left[ \begin{array}{ccc}
1 & 0 & -y/2 \\
0 & 1 & x/2
\end{array} \right]\vec{n} \,
  \right| \right|_{Q^*}  d \sigma, \text{ where } \vec{n} = \frac{\nabla F}{||\nabla F||} \text{ is the unit normal vector}$$ 
\end{theorem}

With this reduction to integrals of Minkowski content, we can take surface area to be either the Minkowski content itself 
or the Minkowski content in the `dual' CC-metric $d_{Q^*}$ arising from the anti-norm $|| \cdot||_{Q^*}$, which I call the {\em anti-Minkowski content}.
In either case, applying the above theorem yields a reduction to integrals, 
and the first variation of perimeter yields a result for sufficiently smooth surfaces in polygonal CC-metrics found in Section \ref{subsec:polymeancurv}.

\begin{theorem}[Mean curvature of smooth surfaces, polygonal case]\label{thm:meancurvatureismean}
Say $F$ is a $C^2$ smooth function that parametrizes a surface $S$ in $(H(\R), d_Q)$ where $Q$ is a polygon.
Then for either definition of surface area, the mean curvature distribution of $F$ is supported on a union of $C^1$ curves and isolated points,
and hence is almost-everywhere $0$.
\end{theorem}

We can then observe that in the polygonal setting, having mean constant curvature almost-every $0$ is quite common. 
One might suspect surfaces with Legendrian foliation by geodesics are isoperimetrices as per Pansu's conjecture. 
In later sections such surfaces are constructed for general $Q$ and called {\em $Q$-bubble sets}.
The integral formulas can be used to preform calculations showing they are not isoperimetrically optimal.  
The following is observed in Section \ref{subsec:bestconst}: 

\begin{proposition}[Pansu bubble sets appear optimal]\label{thm:pansuisgood}
Let $(H(\R), d_Q, \lambda)$ be the sub-Finsler metric measure space given by
the Heisenberg group $H(\R)$  equipped with $d_Q$ the CC-metric arising from the norm $||\cdot||_1$ ($Q$ the unit diamond) and  either notion of surface area. 
Then the isoperimetric ratio of the Pansu bubble set is higher than that of the the diamond bubble set, the square bubble set, and the metric spheres for $d_Q$ and $d_{Q^*}$.
\end{proposition}

This behavior is completely unlike the case of Finsler geometry, where Riemannian isoperimetrices are not Finsler isoperimetrices.
With that in mind, I make the following conjecture regarding the isoperimetric profile 
of the sub-Finsler Heisenberg groups in relation to the sub-Riemannian case.

\begin{conjecture}[Generalized Pansu Conjecture]\label{conj:genpansu}
For $(H(\R), d_Q, \lambda)$ with any CC-metric $d_Q$ and either notion of surface area, 
the Pansu bubble set is the unique isoperimetrix (up to dilation and translation).
\end{conjecture}

\subsection{Acknowledgments}
The author would like to thank Moon Duchin for overseeing this project (and my related dissertation work as found in \cite{thesis})
as well as Luca Capogna, Enrico Le Donne, Sebastiano Nicolussi Golo, Fulton Gonzalez, Meng-Che Ho, Kim Ruane,  Michael Shapiro, and Hang Lu Su 
for conversations and ideas that helped contribute to this work.

\section{Background} \label{sec-back}
\subsection{Finsler geometry in $\R^2$}
Normed planes are simple examples of Finsler geometries that are central to the study of convex geometry (See \cite{Vershynin} for details) 
and are important to the study of sub-Finsler perimeter in the Heisenberg group.
A norm in $\R^2$ must have unit ball a closed, centrally symmetric convex plane figure $Q$.
Conversely, for a such a $Q$, one can define a norm $|| \cdot ||_Q = \inf\{ a \mid v \in aQ\}$,
and hence choice of a planar norm is equivalent to choice of a unit ball $Q$.
We can take the dual space of $(\R^2 , ||\cdot||_Q)$, and apply the Riesz representation theorem to conclude 

\begin{lemma}[Anti-norms]
The dual space of $(\R^2, ||\cdot||_Q)$ is isomorphic to $(\R^2, ||\cdot||_{Q^*})$,
where $||y||_{Q^*} := \max_{x \in Q}\{ \langle x, y\rangle \}$.
\end{lemma}

\begin{proposition}[Anti-norms are dual norms]
The unit ball of $||\cdot||_{Q^*}$ is a convex centrally symmetric closed plane figure given by the polar dual $Q^* = \{ y \in \R^2 \mid \langle x, y\rangle \leq 1 \ \ \forall x \in Q\}$,
 and hence $||\cdot||_{Q^*}$ is also called the dual norm of $||\cdot||_Q$.
\end{proposition}

In keeping with the fact that the dual norm is also sometimes called the `anti-norm', we will use the terminology that for the norm $||\cdot||_Q$, 
$Q$ is the {\em ball},  $Q^*$ is the {\em anti-ball}, $L = \partial Q$ is the {\em sphere}, and $L' = \partial Q^*$ is the {\em anti-sphere}.
Recall that in the plane, for $1 \leq p < \infty$ a $p$-norm is $||\cdot||_p = \left(  x^p + y^p \right)^{1/p}$, where $\displaystyle ||\cdot||_2 = ||\cdot||_{\D} = ||\cdot|| = \sqrt{\langle \cdot, \cdot \rangle}$ is the the Euclidean norm induced by the Euclidean inner product ($\D$ the unit ball).  Similarly, define $||\cdot||_{\infty} = \sup\{x, y\}$.

\begin{example}
$||\cdot||_p$ and $||\cdot||_q$ are dual iff $\frac{1}{p} + \frac{1}{q} = 1$, and $||\cdot||_1$ is dual to $||\cdot||_{\infty}$.
\end{example}

The Polar dual is a useful tool from convex geometry, and has the following properties.

\begin{proposition}[Properties of polar duals]
For $Q_1, Q_2$ are closed, centrally symmetric, convex plane figures and $r \in \R$, 
\begin{enumerate}
\item $(Q_1 \cup Q_2)^* = Q_1^*  \cap Q_2^*$
\item $(Q_1 \cap Q_2)^* = Q_1^*  \cup Q_2^*$
\item if $Q_1 \subset Q_2$, then $Q_1^* \supset Q_2^*$
\item $(rQ)^* = (1/r) Q^*$
\item $Q^{**} = Q$ (the bipolar theorem)
\item if $Q$ is a polygon, then $Q^* = \{ y \in \R^2 \mid \langle x, y\rangle \leq 1 \forall x \in v(Q) \}$, where $v(Q)$ is the set of vertices of $Q$. \\
Furthermore, if $Q$ has $n$ vertices and $m$ edges, $Q^*$ has $m$ vertices and $n$ edges.
\item if $Q$ is an equilateral polygon then $Q^*$ is an equiangular polygon. 
\end{enumerate}
\end{proposition}

From this we can also think of our $p$-norm examples for $1$ and $\infty$ via the sixth property,
which tells us that the unit diamond and unit square are dual polygons, confirming that the associated norms are dual.
Furthermore, since the unit disk is self-dual, we have a geometric explanation for why the Euclidean norm $||\cdot||$ is self dual.

From a planar norm, we can define arclength by generalizing the usual notion of Euclidean arclength:
\begin{definition}[Minkowski length]Let $||\cdot||_Q$ be a norm in the plane.
For a parametrized curve $\gamma: [0,T] \to \R^2$ , define the Minkowski length of $\gamma$ with respect to $||\cdot||_Q$ to be 
$$ \int_0^T ||\gamma'(t)||_Q  \, dt. $$
\end{definition}

This notion of arclength gives a metric on $\R^2$ that is, in general a {\em Finsler} metric rather than a {\em Riemannian} one,
meaning that in every tangent space has a norm (not necessarily induced by an inner product), 
length of a curve is given as the integral of the norm of tangent vectors along the curve,
and distances between points is defined as the $\inf$ of lengths of curves between the points.
There isoperimetric question for Minkowski length was solved as such:

\begin{theorem}[The Isoperimetric Inequality in Normed Planes, \cite{BusemannIso}]\label{thm:busemanstheorem}
Given a planar norm $||\cdot||_Q$, the  maximal ratio of area to Minkowski length squared is achieved uniquely (up to scaling) by $\partial I$,
where $I = e^{\pi/2} Q^*$, the rotate of the polar dual by $90$ degrees.
\end{theorem}

\begin{example}Circles are optimal for the usual Euclidean distance (a fact known for centuries), 
squares are optimal for the Minkowski length via $L^1$,  and diamonds are  optimal for the Minkowski length via $L^\infty$.
\end{example}

A second notion of arclength comes from geometric measure theory: The Minkowski content, 
which measures infinitely small neighborhoods around the object and is, in general, not the same as Minkowski length.

\begin{definition}[Planar Minkowski content of a curve] Let $||\cdot||_Q$ be a norm in the plane.
Let $\Gamma$ be a planar figure with boundary curve $\gamma$. Define the Minkowski content of $\gamma$ with respect to $||\cdot||_Q$ to be 
$$ \lim_{\varepsilon \to 0} \frac{\lambda \big( \mathcal{N}^{||\cdot||_Q}_{\varepsilon}( \Gamma) \big) - \lambda(\Gamma) }{ \varepsilon}
=  \lim_{\varepsilon \to 0} \frac{\lambda \big( \mathcal{N}^{||\cdot||_Q}_{\varepsilon}( \gamma) \big) }{2 \varepsilon}, $$
where $ \mathcal{N}^{||\cdot||_Q}_{\varepsilon}( \Gamma)$ is the $\varepsilon$ neighborhood of $\Gamma$ in $||\cdot||_Q$.
\end{definition}

While it is not clear why Minkowski length and Minkowski content are different, it can readily be observed when we describing them in terms
of projections and intersections as per the following remark.

\begin{remark}Let $L \subset (\R^2, ||\cdot||_Q)$ be a line segment.  \label{lem:proj}
Then the Minkowski length of $L$ equals $\frac{2}{|| L \cap Q || }  ||L||$.
Similarly, the Minkowski content of $L$ equals $\frac{||Proj_{L^\perp} (Q)||}{2} || L ||$. 
Hence if you make $Q$ larger, then the Minkowski length of $L$ get smaller while the Minkowski content of $L$ gets larger.
\end{remark}

\begin{figure}[h!]
  \centering
\begin{tikzpicture}
\draw[<->] (-1.5,0) -- (1.5,0);
  \draw[<->] (0,-1.5) -- (0,1.5);
   \draw[blue, line width=1.7pt] (1,0)  -- (0,1) -- (-1,0)  -- (0, -1) -- (1,0);
\draw[blue] (.3,1)  node[right] {$Q$};
\draw[black, line width=1.1pt] (-1.5,-1.5) -- (1.5,1.5); 
\draw[black] (1.2,1)  node[right] {$L$};
\draw[red, line width=1.1pt] (-3,0) --  (0,-3); 
\draw[red] (-2.7,0)  node[right] {$L^\perp$};
\draw[gray] (-1,0) -- (-2,-1);
\draw[gray] (0,-1) -- (-1,-2);
\draw[purple, line width=1.7pt] (-2,-1) --  (-1,-2); 
\draw[purple] (-3.2,-2)  node[right] {$Proj_{L^\perp} (Q)$};
\draw[orange, line width=1.7pt] (-0.5,-0.5) --  (0.5,0.5); 
\draw[orange] (1,0.4)  node[right] {$L \cap Q$};
\end{tikzpicture}
\caption{Minkowski length and Minkowski content of a line can be determined by intersections and projections of the planar figure $Q$ with the line.}
\end{figure}
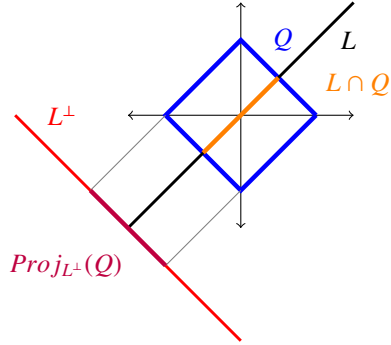

\begin{example}
In $(\R^2, ||\cdot||_1)$, the straight line from the origin to $(1,1)$ has Minkowski length $2$ but Minkowski content $1$.
Note that the norm of the vector from the origin to that point is $||(1,1)||_1 = 2$, which agrees with Minkowski length.
\end{example}

\begin{remark}[Two perimeters in normed planes]
Let $||\cdot||_Q$ be a norm in the plane.  There are two inequivalent notions of perimeter of a compact set.
One is to take the Minkowski length of the boundary curve, and the other is to take the Minkowski content of the boundary curve.
\end{remark}

While it seems odd that in something as simple as a normed plane there can be multiple arclengths, 
it is actually quite standard for Finsler geometry.  
In fact general Finsler spaces have multiple inequivalent notions of volume (ex. Busemann, Holmes-Thompson, and Benson/Gromov Mass*) that are all equivalent when in a Riemannian manifold\cite{JCA,JCA2}.
A priori, it would seem we need to ask the isoperimetric question for this new notion of length, but
there is an explicit relationship between the two given by Busemann as follows.

\begin{theorem}[Minkowski content and Minkowski length are dual, \cite{BusemannGeod}]\label{thm:minkisanti}
The Minkowski content of a curve in $(\R^2 , ||\cdot||_Q)$ is equal to the Minkowski length of the curve in $(\R^2 , ||\cdot||_{Q^*})$.\\
\end{theorem}

\subsection{Sub-Finsler and sub-Riemannian geometry in $H(\R)$}
As stated before, through the paper I will consider the Heisenberg group in {\em exponential coordinates} $H(\R) \cong (\R^3, *)$ , where 
$(x_1, y_1, z_1)*(x_2, y_2, z_2) = (x_1 + x_2, y_1,+ y_2, z_1 + z_2 - \frac{1}{2}(x_1 y_2 - x_2 y_1)$. 
Its Haar measure (left invariant measure which is finite on compact sets) is Lebesgue measure (up to a constant multiple), 
and so for simplicity we can take volume to be Lebesgue measure $\lambda$.
In coordinates, $H(\R)$ has {\em dilations} $\delta_s (x,y,z) = (sx, sy, s^2z)$, 
making $\lambda(\delta_s(S)) = s^4 \lambda(S)$. 
($H(\R)$ has {\em homogeneous dimension} $4$).
$H(\R)$ is a Carnot group, meaning it has a nicely graded Lie algebra. In coordinates, we have the left invariant frame
$$ 
X  = \frac{\partial}{\partial x} - \frac{y}{2} \frac{\partial}{\partial z} , \ \ \ \ \ \ 
Y  = \frac{\partial}{\partial y} + \frac{x}{2} \frac{\partial}{\partial z} , \ \ \ \ \ \ 
Z = \frac{\partial}{\partial z},$$
which gives us the {\em horizontal plane field} $\mathcal{H}$. At each point $p \in H(\R)$, $\mathcal{H}_p$ is $span(X_p, Y_p) \subset T_p H(\R)$.
We use the horizontal plane field to define {\em regular surfaces} to be a surfaces such that the set of points where $\mathcal{H}_p$ is tangent to the surface 
(the characteristic set) is measure zero.
Furthermore, we define the {\em Legendrian foliation} of a regular surface $S$ to be the foliation achieved by taking integral curves to the vector field $\mathcal{H}_p \cap T_pS$.

\begin{figure}[h!]
\centering
\includegraphics[scale = 0.4]{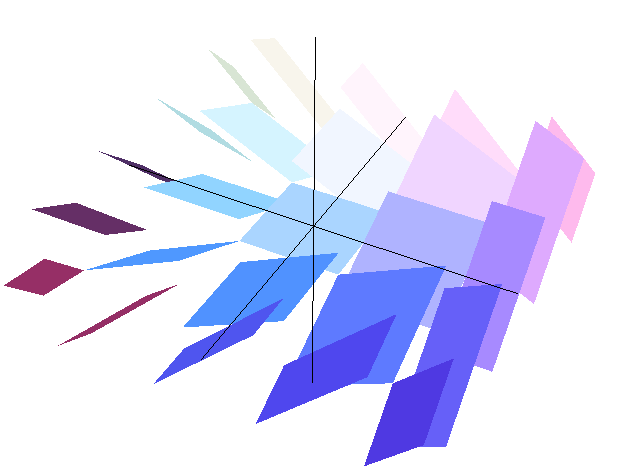}
\caption{ The Horizontal plane field $\mathcal{H}$ in $H(\R)$.}
\end{figure}

We say a curve $\gamma$ is {\em admissible} if it at every point $\gamma'(p) \in \mathcal{H}_p$, and only define arclength for such curves as follows:
\begin{definition}[CC-metrics]
For a planar norm $||\cdot||_Q$, define the Carnot-Carath\'eodory (CC) metric $d_Q$ by  
$$d_Q(\alpha, \beta) := 
\inf \left\{  T \mid \exists \gamma : [0,T] \to \R^3 \text{ admissible, } \gamma(0) = \alpha, \gamma(T) = \beta, \text{ and } ||(\gamma_1'(t), \gamma_2'(t))||_Q \leq 1  \right\}.$$
\end{definition}

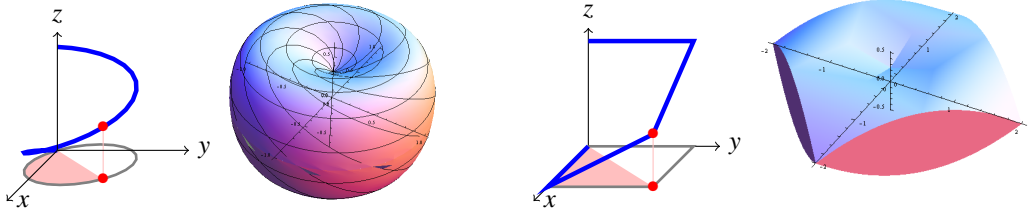
\begin{figure}[h!]
\centering
\begin{tikzpicture}[domain=.7854:6.28 + .7854, scale =0.7]

\filldraw[fill=pink]  plot ({.707+cos((-\x - .7854 -3.14)/2  r)},0, {.707 + sin((-\x - .7854 -3.14)/2 r)} );
  \draw[gray, line width=1pt] plot ({.707+cos((\x -3.14)  r)},0, {.707 + sin((\x  -3.14) r)} );
  \draw[blue, line width=1.7pt] plot ({.707+sin((\x -3.14)  r)},{ (\x - .7858 - sin(\x - .7854) ) /3.14 }, {.707 + cos((\x  -3.14) r)} );

  \draw[pink] ({.707+sin((.7854)  r)},{ (3.14 - sin( .7854) ) /3.14 }, {.707 + cos((.7854) r)} )  -- ({.707+sin(( .7854)  r)},0, {.707 + cos(( .7854) r)} ) ;

\node [red] at ({.707+sin((.7854)  r)},{ (3.1 - sin( .7854) ) /3.14 }, {.707 + cos((.7854) r)} ) {\textbullet};
\node [red] at ({.707+sin(( .7854)  r)},0, {.707 + cos(( .7854) r)} ) {\textbullet};

  \draw[->] (0,0,0) -- (0,0,2.5) node[right] {$x$};
  \draw[->] (0,0,0) -- (2.5,0,0) node[right] {$y$};
  \draw[->] (0,0,0) -- (0,2.25,0) node[above] {$z$};

\end{tikzpicture} 
\includegraphics[scale = 0.25]{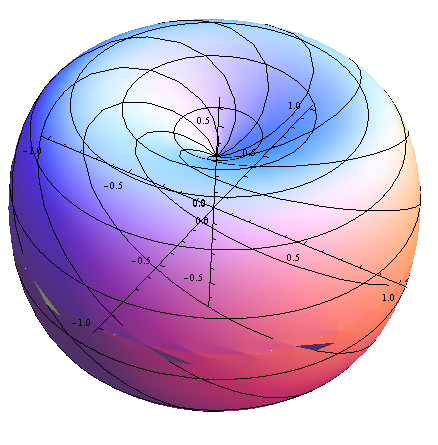}
\hspace{.4in}
\begin{tikzpicture}[domain=.7854:6.28 + .7854, scale =0.7]
  \draw[->] (5,0,0) -- (5,0,2.7) node[right] {$x$};
  \draw[->] (5,0,0) -- (7.5,0,0) node[right] {$y$};
  \draw[->] (5,0,0) -- (5,2.25,0) node[above] {$z$};

\filldraw[fill=pink]  (5,0,0) -- (5,0,2) -- (7,0,2);
  \draw[gray, line width=1pt] (5,0,0) -- (5,0,2) -- (7,0,2) -- (7,0,0) -- (5,0,0);
  \draw[blue, line width=1.7pt] (5,0,0) -- (5,0,2) -- (7,1,2) -- (7,2,0) -- (5,2,0);

\draw[pink] (7,1,2) -- (7,0,2);
\node [red] at (7,0,2) {\textbullet};
\node [red] at (7,1,2) {\textbullet};
\end{tikzpicture}
\includegraphics[scale = 0.3]{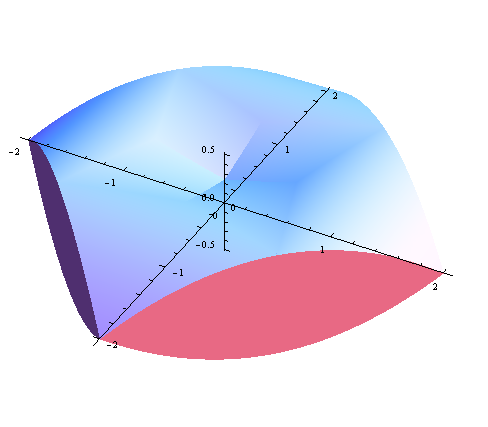}
\caption{ 
Geodesics and metric spheres for the CC-metrics that arise from $||\cdot||_2$ (sub-Riemannian)  and $||\cdot||_1$ (sub-Finsler) norms, shown on the left and right, respectively.
\label{bubbles}}
\end{figure}

We say that such a metric is {\em sub-Finsler} in that  not all paths have finite length, but when they are, the metric is Finsler. 
The case of the Euclidean norm is called the {\em sub-Riemannian} case.
While this may appear quite unwieldy, in exponential coordinates CC-metrics have a nice geometric interpretation. 
First we note a fact about admissible curves:

\begin{theorem}[Path lifting]
$(\gamma_1(t), \gamma_2(t), \gamma_3(t)) \in H(\R)$ is admissible iff  $\gamma_3(t)$ is the area enclosed by $ (\gamma_1(t), \gamma_2(t)) \in \R^2$.
\end{theorem}

We can then note an easy geometric interpretation of CC-metrics.

\begin{theorem}
The Minkowski length of $(\gamma_1(t), \gamma_2(t))$ in $(\R^2, ||\cdot||_Q)$ is equal to the CC-length $d_Q$ of its lift.
\end{theorem}

Thus we can describe geodesics in terms of planar paths and the planar isoperimetric problem.
\begin{theorem}[CC-Metric Geodesics \cite{M&M}]
A curve from $(x_1, y_1, z_1)$ to $(x_2, y_2, z_2)$ in  $(H(\R), d_Q)$ is geodesic iff its $xy$ projection is a path 
from $(x_1, y_1)$ to $(x_2, y_2)$ enclosing area $z_2 - z_1$ whose Minkowski length is minimal among such paths.
It follows that scaled and translated subarcs of the boundary of the planar isoperimetrix $\partial I$ lift to geodesics.
\end{theorem}

For the sub-Riemannian case, consider Riemannian approximations $(\R^3, g_L)$, where in exponential coordinates,
$$g_L = \left[\begin{array}{rrr} 1 & 0 & 0\\ 0 & 1 & 0 \\ -\frac{1}{x_2} & \frac{1}{2}x_1 & 1 \end{array} \right]^T 
\left[\begin{array}{rrr} 1 & 0 & 0\\ 0 & 1 & 0 \\ 0 & 0 & L \end{array} \right] 
\left[\begin{array}{rrr} 1 & 0 & 0\\ 0 & 1 & 0 \\ -\frac{1}{x_2} & \frac{1}{2}x_1 & 1 \end{array} \right].$$ 
Note that in these Riemannian metrics, the length of a path in the $z$ direction is $L$.
As $L \mapsto \infty$, $g_L$ tends towards the sub-Riemannian metric. 
In $(\R^3, g_L)$ there is a Riemannian surface area given by Minkowski content, and the limit of Minkowski content of a regular surface in an approximant 
is the Minkowski content in the sub-Riemannian case.
This gives a definition of surface area of a regular surface $\Sad(S)$ for the sub-Riemannian case with nice analytic properties.
Furthermore, in the sub-Riemannian case Pansu constructed a {\em bubble set} by taking all geodesics from the origin to $(0,0,1)$, which is the unique surface
(up to translation/dilation) that has Legendrian foliation by geodesics, which, in particular, implies the mean curvature is constant.
Pansu's {\em bubble set} is a topological sphere, and hence encloses some region.
Pansu conjectured such a bubble set to be the solution to the sub-Riemannian isoperimetric problem.
\begin{conjecture}[Pansu's Conjecture, \cite{Pansu2,Pansu1}]\label{conj:pansu}
For $H(\R)$ with the sub-Riemannian CC-metric, for all surfaces $S$ enclosing region $E$, the isoperimetric ratio satisfies $\lambda^{3/4}(E) / \Sad(S)  \leq C^{iso} =  3^{3/4} / 4 \sqrt{\pi} =0.3215\ldots \ \  $, \,
with equality holding iff $S$ is (up to dilation/translation) the Pansu bubble set. \, I.e. The bubble set is the unique isoperimetrix.
 \end{conjecture}

Pansu's conjecture remains open in general, but there are substantial partial results. In particular, 
metric spheres are known not to be isoperimetrically optimal \cite{Monti3}, and the conjecture holds if
you restrict to the class of $C^2$ smooth sets \cite{RR2}, radially symmetric sets \cite{Monti1}, or convex sets \cite{MR}.
See \cite{Capogna} for an in-depth study.

\section{Two Sub-Finsler Perimeter Measures} \label{sec-sa}
\subsection{Minkowski content and anti-Minkowski content} 
In the sub-Riemannian case, there is a unique definition of surface area that is equal to the Minkowski content,
and agrees with Riemannian approximations. 
More generally, a perimeter measure should be compatible with the Carnot group structure (dilations/translations).

\begin{definition}[K-Homogeneity]
A measure $\mu$ is said to be $K$-homogeneous if $\forall E$ measurable, $\mu(\delta_s E) = s^K \mu(E)$ .
\end{definition}
We note that in the Heisenberg group Haar measure is $4$-homogeneous.
Furthermore the sub-Riemannian surface area measure is $3$-homogeneous.
We generalize this by defining a perimeter measure as follows:

\begin{definition}[Perimeter measure]
For the metric measure space $(H(\R), d_Q, \lambda)$,
a perimeter measure $\mu$ is a left-invariant $3$-homogeneous ($\mu(\delta_s  S) = s^3 \mu(S)$)
 measure on regular surfaces in $H(\R)$  that is equal to the sub-Riemannian perimeter measure in the case of $Q = \D$, the unit disk.
\end{definition}

From geometric measure theory, we have Minkowski content, which is a left-invariant $(k-1)$-homogeneous perimeter measure in any metric, measure space where $k$ is the homogeneous dimension. 

\begin{definition}[Heisenberg Minkowski content]\label{def:Mink}
Let $E \subset (H(\R), d_Q, \lambda)$ be a compact region bounded by the regular surface $S$.  Then the Minkowski content of $S$ is  
$$ \Sa(S) := \lim_{\varepsilon \to 0}  \frac{\lambda(\Neigh{E}) - \lambda(E)}{\varepsilon} = \lim_{\varepsilon \to 0}  \frac{\lambda(\Neigh{S})}{2 \varepsilon},$$
Where $\Neigh{S}$ is the $\varepsilon$ neighborhood of $S$ in $d_Q$
\end{definition}

Minkowski content behaves nicely with respect to metric balls, because the neighborhood of a metric ball is just a dilation of the unit ball, and we have the following simple fact.

\begin{proposition}The Minkowski content of a metric sphere of radius $r$ $\displaystyle \partial (\B_r )$ is $\Sa(\partial (\B_r )) = 4 r^3 \lambda(\B_1) = \frac{4}{r} \lambda(\B_r)$. \label{prop:balls}
\end{proposition}
\begin{proof}
We can think of $\B_r$ as $\delta_r (\B_1)$. A nice property of metric balls is that they scale, so $\Neigh{\B_r}$ is the ball of radius $(r + \varepsilon)$,
and has volume $(r + \varepsilon)^4 \lambda(\B_1)$ since it is just a dilation.  This makes the Minkowski content
$$ \lim_{\varepsilon \to 0} \frac{ (r^4 + 4 \varepsilon  r^3 + 6 \varepsilon^2 r^2 + 4 \varepsilon ^3 r + \varepsilon^4 r)(\lambda(\B_{1})) - r^4 \lambda(\B_{1}) }{\varepsilon } = 4 r^3 \lambda(\B_1). $$
\end{proof}

\begin{remark}[Minkowski content agrees with the discrete group] The metric ball property above makes Minkowski content consistent with $H(\Z) < H(\R)$ as follows:
Say $\mathcal S$ is a symmetric generating set for $H(\Z)$, and let $Q$ be the convex hull of the projection of $\mathcal{S}$ to the $xy$ plane,
$Q = CHull( \pi(\mathcal{S}))$.  Then cumulative growth function has leading term $\lambda(\B_{1}) n^4$ and  
the spherical growth function has leading term $4 \lambda(\B_{1}) n^3 = \Sa(\B_{d_Q}) n^3$ \cite{Pansu3, M&Mike}.
\end{remark}

Recall, however, that in normed planes, it was not Minkowski content but rather Minkowski length that seemed to be a best analog to Euclidean arclength in that it agreed with the norm on straight line paths.  Minkowski length was defined by taking neighborhoods in the anti-norm (also called the dual norm).  
As such we can take a Heisenberg analog to Minkowski length by taking neighborhoods in the CC-metric that arises from the anti-norm $d_{Q^*}$, which I call anti-Minkowski content.

\begin{definition}[Heisenberg anti-Minkowski content]\label{def:Mink}
Let $E \subset (H(\R), d_Q, \lambda)$ be a compact region bounded by the regular surface $S$.  Then the anti-Minkowski content of $S$ is  
$$ \ASa(S) := \Sad_{Q^*}(S) = \lim_{\varepsilon \to 0}  \frac{\lambda(\ANeigh{E}) - \lambda(E)}{\varepsilon} = \lim_{\varepsilon \to 0}  \frac{\lambda(\ANeigh{S})}{2 \varepsilon}.$$
\end{definition}

\subsection{An integral formula for Minkowski content}\label{subsec:thm1}
Generally, computing Minkowski content is difficult because neighborhoods in a CC-metric are complicated.
It is convenient to begin by considering a surface given by a the graph of compactly supported functions in the $xy$ plane.

\begin{proposition}[The boundary of the $\varepsilon$ neighborhood of a function] For the surface given by $f(x,y)$ above the region $\Omega \subset \R^2$, the $\varepsilon$ neighborhood of the surface has top boundary given by the graph of the function 
$$Z_\varepsilon (x, y) = \max_{(a,b) \in Q} \left. \left\{  f(x - \varepsilon a, y - \varepsilon b) + \frac{\varepsilon}{2} (xb - ya) + {\varepsilon^2} ( H_N(a,b) ) \, \right|  (x - \varepsilon a, y - \varepsilon b) \in \Omega \right\}.$$
\end{proposition}

\begin{proof}
$\text{Let }  E := \{ (x,y,z) \mid  x,y \in \Omega,  0 < z \leq f(x,y) \} \, \text{ and }\, \partial^+E :=  \{ (x,y,z) \mid  x,y \in \Omega,  z = f(x,y) \}$.\\
$\Neigh{E}$ is then the union of copies of the $\varepsilon-$ball translated to every point on $E$, 
which makes the neighborhood $\Neigh{E} = E \bigcup_{(x,y,z) \in \partial E} L_{(x,y,z)}  \B_\varepsilon $.
It is clear then that if we let $\Omega = \supp(f) \subset \R^2$, so that $\Neigh{E}$ should be bound by functions on a (planar) neighborhood of $\Omega$,
and that the height above a point $(x,y) \in \Omega$ should be the maximum height achieved by the translate of a unit ball.
This can be then also be described as the graph of a well-defined function, because at every point it is the maximum of a function over a closed region.
The extremal points of the ball are on the sphere.
For a point $(a, b, H_Q(a,b) )$ in the unit sphere, if we translate to end above $(x, y)$, we have the following 

$$ (x - a, y - b, f(x - a, y - b))  *  (a, b, H_Q(a,b) ) = (x, y, Z(x,y,a,b)),$$
$$ \text{ where }  Z(x,y,a,b)) =  f(x - a, y - b) + H_Q(a,b)  + \frac{(x - a)(b)}{2} - \frac{(y - b)(a)}{2}.$$

Now note that points on the $\varepsilon$-sphere are $\delta_{\varepsilon}(a,b,H_N(a,b)) = (\varepsilon a, \varepsilon b, \varepsilon^2 H_Q(a,b)$, 
and we have that a point above $(x,y)$ achieved by left translating the $\varepsilon$ sphere is $(x,y, Z(x,y, \varepsilon x, \varepsilon y) )$, with

$$  Z(x,y, \varepsilon a, \varepsilon b) =  f(x - \varepsilon a, y - \varepsilon b) + \varepsilon^2 H_Q(a,b)  + \frac{(x - \varepsilon a)(\varepsilon b)}{2} - \frac{(y - \varepsilon b)(\varepsilon a)}{2} $$
$$\text{ which simplifies to }  f(x - \varepsilon a, y - \varepsilon b) + \frac{\varepsilon}{2} (xb - ya) + \varepsilon^2 ( H_Q(a,b) )  $$

To find the best of these above a particular $(x,y) \in \Omega$, we maximize over all possible nearby points on the surface and all possible points on the sphere, which is given by
$$Z_\varepsilon (x, y) = \max_{(a,b) \in Q} \left. \left\{  f(x - \varepsilon a, y - \varepsilon b) + \frac{\varepsilon}{2} (xb - ya) + {\varepsilon^2} ( H_Q(a,b) ) \,  \right| (x - \varepsilon a, y - \varepsilon b) \in \Omega \right\}, $$

which is the maximum of a continuous function over a closed region, and thus is well defined.
\end{proof}

Geometrically, we can think of each point $(x,y,z)$ on the surface as the endpoint of some admissible curve achieved by lifting a curve in the plane with endpoint $(x,y)$.
Then points on the boundary of the neighborhood are achieved by taking a curves with endpoints on the surface near $(x,y,z)$  
concatenated with geodesics that have endpoints projecting to $(x,y)$.
The surface that cuts out the top of $\Neigh{E}$ is then made up of the best ways to continue nearby curves by geodesics.
We now have a theorem that gives us an integration formula for surface area.
\begin{theorem}[Minkowski content for functions]\label{thm:funct} The Minkowski content of the the graph $S$ of a $C^1$ smooth function $f(x,y)$ above the bounded region $\Omega$ in is given by
$$ \Sa(S) = \iint_{\Omega}  ||( -y/2 , x/2 ) - \nabla f(x,y)||_{Q^*}  d \lambda$$ 
\end{theorem}
\begin{proof}
To find surface area, you take $\Sa(\partial E^+)$.
From the previous proposition, we have 
$ \lambda( E_\varepsilon ) =  \iint_{\Omega_\varepsilon} Z_\varepsilon (x, y) d\lambda$, and also $ \lambda(E)=\iint_{\Omega} f(x, y) d\lambda $.
From the definition of Minkowski content,  we should subtract those and divide by $\varepsilon$, which yields
 $ \iint_{\Omega } Z_\varepsilon (x, y) - f(x,y) d\lambda + \iint_{\Omega_\varepsilon \backslash \Omega } Z_\varepsilon (x, y) d\lambda$.
 The second term  gives the surface area contribution of $z$ above $\partial \Omega$, which is a curve in $\R^3$, and thus has measure zero.
We should then have 
$$\Sa(\partial E^+) = \lim_{\varepsilon \to 0} \frac{ \iint_{\Omega} Z_\varepsilon (x, y) - f(x,y) d\lambda  }{\varepsilon}= \lim_{\varepsilon \to 0} \iint_{\Omega}  \frac{ Z_\varepsilon (x, y) - f(x,y)  }{\varepsilon} d\lambda  $$

Now since $f(x,y)$ and $\varepsilon$ do not depend on $(a,b)$, we can slide them into the $\max$, which yields 
$$\frac{  Z_\varepsilon (x, y) - f(x,y)}{\varepsilon} = \max_{(a,b) \in Q} \left\{ \left. \frac{f(x - \varepsilon a, y - \varepsilon b) - f(x,y)}{\varepsilon} + \frac{1}{2} (xb - ya) + {\varepsilon} ( H_Q(a,b) )  \, \right|   (x - \varepsilon a, y - \varepsilon b) \in \Omega \right\}. $$

Now use the fact that the unit ball $H_Q(a,b)$ is uniformly bounded by some $K$ where $|H_Q(a,b)| \leq K$, and we have

$$\frac{  Z_\varepsilon (x, y) - f(x,y)}{\varepsilon} \leq   \max_{(a,b) \in Q } \left. \left\{  \frac{f(x - \varepsilon a, y - \varepsilon b) - f(x,y)}{\varepsilon} + \frac{1}{2} (xb - ya) \, \right|   (x - \varepsilon a, y - \varepsilon b) \in \Omega  \right\} + {\varepsilon}K$$
$$ \frac{  Z_\varepsilon (x, y) - f(x,y)}{\varepsilon} \geq   \max_{(a,b) \in Q  } \left. \left\{  \frac{f(x - \varepsilon a, y - \varepsilon b) - f(x,y)}{\varepsilon}  + \frac{1}{2} (xb - ya)  \, \right| (x - \varepsilon a, y - \varepsilon b) \in \Omega   \right\} - {\varepsilon}K$$

Since we will take $\varepsilon \to 0$, it is clear that we then only need to deal with what is being maximized.
Observe now that the first part of the inner term is the directional derivative, 

$$\lim_{\varepsilon \to 0} \frac{f(x - \varepsilon a, y - \varepsilon b) - f(x,y)}{\varepsilon} = -\nabla_{(a,b)}  f(x,y) = -\langle (a,b), \nabla f(x,y) \rangle$$

so as $\varepsilon \to 0 $, the max term tends towards $\max_{(a,b) \in Q }\{  \langle u , (  ( -y/2 , x/2 ) - \nabla f(x,y) ) \rangle$.
We may switch the order of the limit and integral because $\overline{\Omega}$ is compact and a continuous function over a compact region is uniformly continuous, 
which makes the integrand uniformly continuous.
The final step is to recall that in the plane, the dual norm is defined as $||y||_{Q^*} = \max_{x \in Q}\{ \langle x , y\rangle \}$, and we have our formula as stated.
\end{proof}

We can now move on to another class of simple surfaces:  vertical planes.

\begin{theorem}[Minkowski content for walls]
Let $S$ be the surface $S := \{ (x,y,z) \mid (x,y) \in L, 0 \leq z \leq h  \}$ for $L$ some line segment in the $xy-$plane.
The Minkowski content of of $S$ is given by 
$$\Sa(S) = ||L||_{Q^*}   h.$$
In particular, the surface area of a vertical rectangle is equal to the product of the planar Minkowski content of the $xy$-component in $(\R^2, ||\cdot||_Q)$ times the height.
\end{theorem}

\begin{proof}
To prove this, we bound $\Neigh{S}$ above and below and squeeze to get what we want.
For ease of notation, denote $k_1  =||  Proj_{L^\perp} (Q) || $ and $k_2 = || Proj_{L} (Q) ||$.
Again, we consider $Z_\varepsilon(x,y)$, the function top bounding function of $\Neigh{S}$ that gives a max height $z$ above a point $(x,y) \in \mathcal{N}_{||\cdot||_Q,\varepsilon}(L)$.

To bound from above, we put the $\varepsilon$ neighborhood of the $L$ inside a box 
parallel to it of side lengths $\varepsilon k_1$ and $||L|| + \varepsilon k_2$.
The unit ball is a convex polygon of finite area,  there is some $k_3$ such that $|a|, |b| \leq k_3$ for all $a,b \in Q$ .
With this we have that $(xb - ya) /2 \leq k_3 (|x| + |y|)$, which is continuous, and hence has a maximum on the box, which we call $k_4$.
Similarly, the unit ball must lay inside some cube with height $k_5$, so we can bound $H_Q(a,b)$,
meaning that $Z_\varepsilon(x,y) \leq h + \varepsilon k_4 + \varepsilon^2 k_5$ for all $x,y$ in the bounding box.
Lastly, the lowest point of the neighborhood must be bounded below by $0 -  \varepsilon k_4 + \varepsilon^2 k_5$.
Thus the volume of the $\varepsilon$ neighborhood of the plane is bounded above by 
$(\varepsilon k_1)  \cdot (||L|| + \varepsilon k_2) \cdot ( h + \varepsilon 2 k_4 + \varepsilon^2 2 k_5)  \leq \varepsilon (k_1)(||L||)(h) + k_6 \varepsilon^2$.

On the other hand, we can take a parallelogram around $L$ that is inside of $\Neigh{S}$, with side length $||L||$ and width $\varepsilon k_1$.
We can then similarly take $h - \varepsilon k_4 - \varepsilon^2 k_5$ and  $ 0 - \varepsilon k_4 - \varepsilon^2 k_5$ as the upper and lower bounds 
of the $z$ values achieved in this parallelogram, and now get
the volume of the $\varepsilon$ neighborhood of the plane is bounded below by  $\varepsilon (k_1)(||L||)(h) - k_7 \varepsilon^2$.

From the second formulation of the Minkowski content definition, these bounds yield 
that the surface area is equal to $\left(\frac{  || \ell (Proj_{L^\perp}(Q)) ||  }{2} \right)  \cdot ||L|| \cdot h$.
The two terms are precisely equal to the Minkowski content of the line segment $L$ in the normed plane with unit ball $Q$ (Lemma \ref{lem:proj}), 
 which makes the first two terms equal to the dual-norm of $L$  as desired.
\end{proof}

From this, we have a nice characterization of surface area of vertically oriented flat regions.

\begin{corollary}[Sub-Finsler Minkowski content for vertical regions]\label{cor:vert}
Let $L$ be a line segment in the $xy$ plane and let $S$ be contained in the plane $\{ (x,y,z) \mid (x,y) \in L  \}$ .
The Minkowski content of $S$ is given by 
$$ \Sa(S) = \left(\frac{  || \ell (Proj_{L^\perp}(Q)) ||  }{2} \right)  \sigma(S) =  \left(\frac{ ||L||_{Q^*}  }{||L||} \right)  \sigma(S),$$
where $\sigma(S)$ is the euclidean surface area of $S$.
\end{corollary}
\begin{proof}
Consider an infinitesimally small rectangle $\Delta z \times \Delta l$ on $S$. 
The euclidean area element on the plane is height in the $z$ direction times the length measured along the direction of $L$, 
so we can express it as $\sigma(S) = ||\Delta l|| ||\Delta z||$.
Reversing that, we have that $ ||\Delta z|| = \sigma(S) / ||\Delta l||$.  
Note now that we can substitute this into the equation for Heisenberg surface area to get $||\Delta l||_{Q^*}   ||\Delta z|| = ||\Delta l||_{Q^*}  ( \sigma(S) / ||\Delta l|| ) $.
Lastly notice that the ratio $(||\Delta l||_{Q^*} / ||\Delta l|| )$ depends only on the line that $l$ lies on, and therefore at every point it is the same constant.
Ergo, at every point the surface area element is $ ||\Delta l||  \cdot ||\Delta z|| (||L||_{Q^*} / ||L|| )$ as desired.
\end{proof}
In particular, the surface area is equal to Euclidean surface area times a scaling factor that depends on the shape of the unit ball and the direction of the plane.

Combining our information about functions and walls, we can describe this surface area for parametric surfaces.

\newtheorem*{result1}{Theorem \ref{cor:MinkFormula}}
\begin{result1}[Sub-Finsler Minkowski content for implicit surfaces] 
Let $(H(\R), d_Q, \lambda)$ be the sub-Finsler metric measure space given by
the Heisenberg group $H(\R)$  equipped with $d_Q$, a CC-metric arising from the norm $||\cdot||_Q$ 
where $Q$ is a closed, centrally symmetric, convex plane figure $Q$,
and $\lambda$ the 3-dimensional Lebesgue measure.
Then the Minkowski content of a $C^1$ smooth regular surface $S \subset (H(\R), d_Q, \lambda)$ parametrized by $F(x,y,z) = 0$
in exponential coordinates  can be computed in given in terms of the Lebesgue surface area measure $\sigma$ as follows: 
$$  \lim_{\varepsilon\to0} \frac{ \lambda(\Neigh{S}) - \lambda(S)}{\varepsilon} =
\iint_{S}\left| \left|  \, \left[ \begin{array}{ccc}
1 & 0 & -y/2 \\
0 & 1 & x/2
\end{array} \right]\vec{n} \,
  \right| \right|_{Q^*}  d \sigma, \text{ where } \vec{n} = \frac{\nabla F}{||\nabla F||} \text{ is the unit normal vector}$$ 
\end{result1}
\begin{proof}
To show equivalence to the case of functions (Theorem \ref{thm:funct}), 
let $F = z - f(x,y)$.
This has gradient vector $\left(-\frac{\partial f}{\partial x}, -\frac{\partial f}{\partial y}, 1 \right)$, 
and $d\sigma = \left|\left|\left(-\frac{\partial f}{\partial x}, -\frac{\partial f}{\partial y}, 1 \right)\right|\right|  dx dy$.

To show equivalence for vertical regions (Corollary \ref{cor:vert}) , consider $F$ such that $\frac{\partial F}{\partial z} = 0$.
In such a case, the integrand is $|| v ||_{Q^*} / || v ||$ for a vector $v$ laying within the line.
Since this only depends on the line, the integral is $|| L ||_{Q^*} / || L || \sigma(S)$.
\end{proof}

The equivalence to the sub-Riemannian case can now readily be observed, since the unit disk $\D$ is self dual,
and therefore we have exactly the formula for surface area as found in  \cite{Capogna}.
Similarly, from the integral we can see that a change of variables will verify 3-homogeneity.
From the formula given, we have an immediate consequence:

\begin{remark}
In exponential coordinates, Minkowski content is symmetric about the origin but not about $x$ or $y$.
\end{remark}
Swapping $x$ and $y$ can already be seen to not be an an isometry because it does not preserve admissibility, but one
might expect that since the norm is symmetric and that swapping $x$ and $y$ preserves volume it would also preserve surface area. 
Surface area not being preserved can readily be observed in the $L^2$ case from Pansu's formulas, but this definition shows that it is more generally not true,
and also means that one must be careful when computing surface area to not attempt to use $xy$ symmetry of a surface as a shortcut!

If we wish to describe the surface area measure intrinsically (i.e. not with respect to exponential coordinates), 
we can do so by relating it to sub-Riemannian perimeter instead.
Recall that $\vec{n_0} = \nabla_0 \vec{n}$ is the horizontal gradient,
which, in exponential coordinates is equal to $ \left[\begin{smallmatrix}
1 & 0 & -y/2 \\
0 & 1 & x/2
\end{smallmatrix}\right] \vec{n}$, and note the following:

\begin{corollary}[Intrinsic Sub-Finsler Minkowski content]
Let $\Sad$ be the sub-Riemannian surface area measure in the Heisenberg group,
Then the Radon-Nikodym derivative of  Minkowski content in $(H(\R), d_Q, \lambda)$ is 
$$ \frac{\partial \Sa}{\partial \Sad}  = \frac{||\vec{n_0}||_{Q^*}}{|| \vec{n_0}||}.$$
\end{corollary}

In general, the norm given in the formula is hard to compute, as it is defined as a maximum over $Q$ of dot products.
In other words, integrating to get surface area in a CC-metric requires you to maximize over a convex region at every point, which is in general is very difficult.
However, in the case of polygonal CC-metrics, we are aided by the following simple calculus fact:
\begin{lemma}
The maximum of a linear function over a region bounded by a polygon is achieved along the a vertex.
\end{lemma}

Combining this with the definition of dual norm, there is a simple formula in the case of polygonal CC-metrics.

\begin{corollary}[Polygonal CC-surface area for implicit surfaces]\label{cor:polygonalSA}
Let $(H(\R), d_Q, \lambda)$ be the sub-Finsler metric measure space given by
the Heisenberg group $H(\R)$  equipped with $d_Q$, a CC-metric arising from the norm $||\cdot||_Q$ 
where $Q$ is a closed, centrally symmetric, convex polygon $Q$ with vertex set $v(Q)$,
and $\lambda$ the 3-dimensional Lebesgue measure.
Then the Minkowski content of a $C^1$ smooth regular surface $S \subset (H(\R), d_Q, \lambda)$ parametrized by $F(x,y,z) = 0$
in exponential coordinates  can be computed in given in terms of the Lebesgue surface area measure $\sigma$ as follows: 
$$\Sa(S) =  \iint_{S}   \max_{(a,b) \in v(Q)} \left\{  
\left\langle 
(a,b), \ 
\left[ \begin{array}{ccc}
1 & 0 & -y/2 \\
0 & 1 & x/2
\end{array} \right]\vec{n}
\right\rangle 
 \right\} d \sigma,$$
where $\displaystyle \vec{n} = \frac{\nabla F}{||\nabla F||}$ is the unit normal vector and $\sigma$ is the Lebesgue surface area measure.
\end{corollary}

This is especially powerful because unlike before where at every point we needed to maximize over a region in the plane, we now only need to maximize over a finite number of possibilities,
which is computationally very easy.

\subsection{Bounding Isoperimetric Ratios}\label{subsec:bounds}
One thing to investigate is how our two perimeter measures change when we change the norm the CC-metric arises from.
Ideally, relating sub-Finsler to sub-Riemannian should give us the ability to carry over some information from the much more heavily studied sub-Riemannian.
On the other hand, such a relation can also give the possibility of approaching sub-Riemannian problems from the sub-Finsler side.

\begin{lemma}[Minkowski content scales linearly]
$\Sar(S) = r \Sa(S)$.
\end{lemma}
\begin{proof}
This comes for free from the integral formula in Theorem \ref{cor:MinkFormula} and the definition of anti-norm by the following observation:
$|| v||_{(rQ)^*} = \max_{w \in rQ} \{ \langle w, v \rangle \} = \max_{ru \in rQ} \{ \langle ru, v \rangle \} = \max_{u \in Q} \{ \langle ru ,  v\rangle \}  = r \max_{u \in Q} \{ \langle u , v \rangle \} = r ||v||_{Q*}$
\end{proof}

A weaker statement relates containment to inequality of Minkowski content

\begin{proposition}\label{prop:containmentbound}
Consider $d_{Q_1}$ and $d_{Q_2}$ such that $Q_1 \subset Q_2$.
Then for all surfaces $S$, we have $\Sa_{1}(S) \leq \Sa_{2}(S)$.
Furthermore, if $Q_1 \subset Q_2^{\circ}$, then the inequality is strict.
\end{proposition}
\begin{proof}
Anything maximized in $Q_1$ can only be made bigger by passing to $Q_2$.
The furthermore condition follows from the fact that what is being maximized is a linear function, which for a convex set is maximized along its boundary.
If the boundary of $Q_2$ is necessarily larger than the boundary of $Q_1$, then at every point we can extend the vector pointing in the maximizing direction 
to end at a point in $\partial Q_2$.  
This makes the integral strictly larger than $\Sa_{1}$, but is less than or equal to $\Sa_{2}$
since there may be a better point to end along $\partial Q_2$.
\end{proof}

From this we have a scheme to approximate the sub-Riemannian surface area.

\begin{corollary}[Weak sub-Finsler approximation to sub-Riemannian]
Let $Q_n$ be the $2^n$ ($n >1$) sided polygon inscribed within the circle and let $S$ be a surface in $H(\R)$. 
Embedding $S$ successively into the sequence of sub-Finsler metric measure space $(H(\R), d_{Q_n}, \lambda)$ 
preserves volume and has non-decreasing Minkowski content and non-increasing anti-Minkowski content limits both limit to the sub-Riemannian perimeter.
\end{corollary}

We also have another corollary.

\begin{corollary}[Isoperimetric Bounds via Containment]
Consider $d_{Q_1}, d_{Q}$, and $d_{Q_2}$ with $Q_1 \subset Q \subset Q_2$.
For a regular  surface $S$ bounding region $E$, the isoperimetric ratios via Minkowski content $C(S) = \lambda(E)^{3/4}/\Sad(S)$ 
of the three CC-metrics are related by 
$\Iso_{Q_1}(S) \geq \Iso_Q(S) \geq \Iso_{Q_2}(S)$.
Similarly, for the isoperimetric ratios via anti-Minkowski content $\AIso(S) = \lambda(E)^{3/4}/\ASad(S)$  are related by $\AIso_{Q_1}(S) \leq \AIso_Q(S) \leq \AIso_{Q_2}(S)$.
\end{corollary}
\begin{proof}
Lebesgue measure does not depend on the metric, so the numerator of the isoperimetric ratio is the same in all six expressions.
From the proposition above we know that the Minkowski contents are related by $\Sa_{1}(S) \leq \Sa(S) \leq \Sa_{2}(S)$, which.
when we divide gives the statement as desired. For the anti-Minkowski statement, note that taking polar dual reverses inclusion which reverses the inequalities.
\end{proof}

Combining this corollary with the earlier lemma, we can combine this information to 
to say that the isoperimetric ratio in any sub-Finsler case is between constant multiples of the sub-Riemannian cases.
Recalling that in the sub-Riemannian case Minkowski content and anti-Minkowski content are equal we have 

\begin{corollary}[Bounds for sub-Finsler isoperimetric ratio by sub-Riemannian]\label{thm:inscribedcircumscribed}
Say $r\D \subset Q \subset R\D$, 
where $r\D$  and $R\D$ are the inscribed 
and circumscribed circles (respectively) for the planar figure $Q$ .
For a regular surface $S = \partial E$, the isoperimetric ratios are bounded as $ \Iso_{\D}(S)/R \leq \Iso_Q(S) \leq \Iso_{\D}(S)/r$
and $ r \Iso_{\D}(S) \leq \AIso_Q(S) \leq R\Iso_{\D}(S)$ .
\end{corollary}

Taking the $L^1$ norm as an example, we know the diamond is between the circle of radius $1/\sqrt{2}$ and radius 1, 
so $\Iso$ of any surface is between $1$ and $\sqrt2$ times the isoperimetric ratio in the sub-Riemannian metric.
Taking maximums, this yields a statement about the isoperimetric constant.

\begin{corollary}
For any $d_Q$ with $Q \subset \D$, the isoperimetric constant for Minkowski content is at least as high as the sub-Riemannian isoperimetric constant. \end{corollary}

So this gives a way to bound sub-Finsler by sub-Riemannian, but the bounds also work the other way to give an approximation
of sub-Riemannian by sub-Finsler polygonal CC-metrics.

\begin{corollary}[Strong sub-Finsler approximation to sub-Riemannian]\label{cor:strongsubfin}
Let $Q_n$ be a regular $2^n$ ($n>1$) sided polygon inscribed in the circle and $R_n Q_n$ is the smallest dilate containing the circle.
Then $Q_n \subset \D \subset R_n Q_n$, and hence,
for a  regular surface $S$,  $ \min\{ \Iso_{Q_n}(S)/R_n, \AIso_{Q_n}(S) \}  \leq \Iso_{\D}(S) \leq \max\{ \Iso_{Q_n}(S), R_n \AIso_{Q_n}(S) \}$, and the two bounds limit to $\Iso_{\D}(S)$.
\end{corollary}
Note that as $n$ gets higher, $R_n$ gets closer to $1$.

\begin{remark}
These bounding corollaries underscore two new strategies to potentially solving Pansu's conjecture.
The first is to prove that the the bubble set is optimal in each of the Sub-Finsler cases of regular polygons,
and use the bounding formulas and limits to prove it in all cases.
Should Pansu's bubble set not be optimal in some class of surfaces, then in a polygonal CC-metric its isoperimetric ratio will be necessarily higher than the upper bound given by the sub-Riemannian isoperimetric ratio of Pansu's bubble set.  As such, if we wish to find a contradiction to Pansu's conjecture, we can work exclusively in the case of polygonal CC-metrics.
\end{remark}

\section{First Variation of Perimeter} \label{sec-fvp}
\subsection{Distributions and mean curvature}
For surfaces embedded in euclidean space, variation of perimeter yields mean curvature and can be used to construct a condition for minimal surfaces. 
In the sub-Riemannian setting, it can also be done, and is used to observe that Pansu's bubble set has constant mean curvature,
which is a major factor in the reasoning to believe Pansu's conjecture that it maximizes isoperimetric ratio. 
We can attempt to do similar in the more general setting.

\begin{definition}
Let $F$ parametrize a surface $S$ and let $\varphi$ be a $C_0^\infty$ function. 
Then for a choice of surface area measure $\mu$,
the first variation of perimeter is 
$\displaystyle \left. \frac{\partial \mu( S + \varepsilon \varphi) }{\partial \varepsilon} \right|_{\varepsilon = 0} = - \int_S \varphi  \M  = - \langle \varphi, \M\rangle, $
and we say that the distribution $\M$ is the mean curvature of $S$ with respect to $\mu$.
\end{definition} 

\begin{definition}
For a distribution $f$ (generalized function), the distributional derivative (or weak derivative) $f'$ is defined by the relation
$<f',\phi> = -<f, \phi'>$  for every test function $\phi \in C_0^\infty$. 
\end{definition}

All the usual properties of differentiation apply to distributional derivatives, and we will see them in the following section.
Three important functions that are useful due to the presence of absolute values are the following:

\begin{definition} 
The signum (or sign) function $\Sgn$ is a distribution such that $\Sgn(x) = |x| / x$ for $x \neq 0$ and $\Sgn(0) = 0$.\\
The Heavyside step-function $H$ is a distribution such that $H(x) = 0$ for $x < 0$, $H(x) = 1$ for $x > 0$ and $H(0) = 1/2$.
The dirac delta function $\delta$ is a distribution such that $\delta(x) = 0$ for $x \neq 0$ and $\int_{-\epsilon}^{\epsilon} \delta = 1$.
\end{definition}

These three distributions are related as per:
$\displaystyle  \frac{\partial |x|}{\partial x} = \Sgn(x) = 2H(x) - 1, \text{ \, } \frac{\partial H(x)}{\partial x} = \delta(x), \text{ \, and \, } \frac{\partial \Sgn}{\partial x} = 2\delta(x) .$

\subsection{Mean curvature in the Polygonal Case}\label{subsec:polymeancurv}
We start with the case of $Q$ being unit square that defines $||\cdot||_{\infty}$,
and note that $Q^*$ is the unit diamond that defines $||\cdot||_1$.
This gives clean expressions for $\Sa$ and $\ASa$ as integrals of expressions involving absolute values.
$\Sa$ is the simpler of the two expressions and contains terms that are present in $\ASa$, so it is a natural starting point.

\begin{lemma}\label{lem:simple} For $(H(\R), d_Q, \lambda)$ where $||\cdot||_Q = ||\cdot||_{\infty}$, the integral formulas for Minkowski and anti-Minkowski content of a parametrized surface simplify to 
$$\Sa(S) =  \frac{1}{2} \iint_{S}   
\Abs{ \Fx  - \frac{y}{2} \Fz}
+\Abs{ \Fy  + \frac{x}{2} \Fz}  dx dy,$$

$$\ASa(S) =  \frac{1}{2}\iint_{S}   
\Abs{ \Fx  - \frac{y}{2} \Fz}
+\Abs{ \Fy  + \frac{x}{2} \Fz}
+ \Abs{  \Abs{ \Fx  - \frac{y}{2} \Fz}  - \Abs{ \Fy  + \frac{x}{2} \Fz}}
 dx dy.$$
\end{lemma}
\begin{proof}
Consider that from Corollary \ref{cor:polygonalSA} the integrand is the maximum over a finite set of points and one can apply
$\displaystyle \max\{x,y\} = \frac{x + y + | x - y |  }{2}$ repeatedly to get the formulas above.
\end{proof}

For the remainder of this section I will restrict consideration to surfaces that are the graphs of functions  and note that for more general surfaces we would have similar statements with more symbol pushing.

\begin{theorem}
Say $f$ is a $C^2$ smooth function that parametrizes a regular surface $S = \{(x,y,f(x,y) \} \subset (H(\R), d_Q, \lambda)$ where $||\cdot||_Q = ||\cdot||_{\infty}$ so that $||\cdot||_{Q^*} = ||\cdot||_1$.
Then the mean curvature distribution for Minkowski content is given by
$$\M_{\Sa} = \frac{1}{2} \Div \left[ \begin{array}{c}
\Sgn\left(\fx + \frac{y}{2}\right) \\
\Sgn\left(\fy - \frac{x}{2}\right)\end{array} \right] =  \delta\left( \fx  + \frac{y}{2}\right) \fxx + \delta\left(  \fy  - \frac{x}{2}\right) \fyy$$
\end{theorem}
\begin{proof}
Consider varying the graph of compactly supported function $f$ by $f + \varepsilon \varphi$,
where $f,\varphi : \Omega \to \R^3$ with $\varphi |_{\partial \omega} = 0$.  Applying our integral formula, we have 
$$ \frac{\partial}{\partial \varepsilon}  \frac{1}{2} \iint_{\Omega}  \Abs{ \left(\fx + \varepsilon\phix\right)  + \frac{y}{2}} +\Abs{ \left(\fy + \varepsilon\phiy\right)  - \frac{x}{2}}  dx dy $$
$$ =  \frac{1}{2} \iint_{\Omega} \frac{\partial}{\partial \varepsilon}  \Abs{ \left(\fx + \varepsilon\phix\right)  + \frac{y}{2}} + \frac{\partial}{\partial \varepsilon}\Abs{ \left(\fy + \varepsilon\phiy\right)  - \frac{x}{2}}  dx dy.$$
Now we take the derivative (as distributions) of each part to get 

$$ \frac{\partial}{\partial \varepsilon}  \Abs{ \left(\fx + \varepsilon\phix\right)  + \frac{y}{2}}  = \Sgn\left(  \left(\fx + \varepsilon\phix\right)  + \frac{y}{2}\right) \phix$$
$$ \frac{\partial}{\partial \varepsilon}  \Abs{ \left(\fy + \varepsilon\phiy\right)  - \frac{x}{2}}  = \Sgn\left(  \left(\fy + \varepsilon\phiy\right)  - \frac{x}{2}\right) \phiy$$

Piecing those together, we have 

$$ \frac{1}{2}  \iint_{\Omega} \nabla \varphi \cdot \left(\Sgn\left(  \left(\fx + \varepsilon\phix\right)  + \frac{y}{2}\right),   \Sgn\left(  \left(\fy + \varepsilon\phiy\right)  - \frac{x}{2}\right)   \right)   dx dy,$$

and we can evaluate at $\varepsilon = 0$ 

$$  \frac{1}{2} \iint_{\Omega} \nabla \varphi \cdot \left(\Sgn\left( \fx  + \frac{y}{2}\right),   \Sgn\left(  \fy  - \frac{x}{2}\right)   \right)   dx dy.$$

This is almost what we want, except that there is a $\nabla$ on $\varphi$. 
To deal with this, we use integration by parts, recall integration by parts interchanges gradient with divergence as per
$$ \int_{\Omega} \nabla \varphi \cdot v = - \int_{\Omega} \varphi (\nabla \cdot v) + \int_{\partial \Omega} \varphi v $$
Since we have chosen everything to be compactly supported and differentiable, the boundary term is zero, and as such, 
and our expression is 

$$ - \frac{1}{2} \iint_{\Omega} (\varphi) \left( \Div \left[ \begin{array}{c}
\Sgn\left(\fx + \frac{y}{2}\right) \\
\Sgn\left(\fy - \frac{x}{2}\right)\end{array} \right]  \right) dx dy = - \iint_{\Omega} \varphi \M_{\Sa}. $$

To get the expression as a PDE, we can take one more derivative.
\end{proof}

We can do the same now for anti-Minkowski content although the PDE expression becomes too cumbersome to write in a clean manner.
As such, we will suppress it and leave mean curvature expressed as a divergence. 

\begin{theorem}
Say $f$ is a $C^2$ smooth function that parametrizes a regular surface $S = \{(x,y,f(x,y) \} \subset (H(\R), d_Q, \lambda)$ where $||\cdot||_Q = ||\cdot||_{\infty}$ so that $||\cdot||_{Q^*} = ||\cdot||_1$.
Then the mean curvature distribution for anti-Minkowski content is given by
$$\M_{\ASa} = \frac{1}{2} \Div \left[ \begin{array}{c}
\left( 1 - \Sgn\left( \Abs{ \fx + \frac{y}{2}}  -  \Abs{ \fy - \frac{x}{2}} \right) \right)  \left(\Sgn\left(\fx + \frac{y}{2}\right) \right) \\
\left( 1 + \Sgn\left( \Abs{ \fx + \frac{y}{2}}  -  \Abs{ \fy - \frac{x}{2}} \right) \right)   \left( \Sgn\left(\fy - \frac{x}{2}\right)  \right)\end{array} \right]$$
\end{theorem} 
\begin{proof}
Recall that from the simplified expressions in the previous chapter, $\ASa$ is equal to $\Sa$ plus the integral of 
$- \Abs{  \Abs{ \fx  + \frac{y}{2}}  - \Abs{ \fy  - \frac{x}{2}}}$.  
As such, when we perturb $f$ by $\varepsilon \varphi$ and take $\frac{\partial}{\partial \varepsilon}$,  we need to add in an extra terms of the form
$$ - \Sgn\left( \Abs{ \left(\fx + \varepsilon\phix\right)  - \frac{y}{2}}  -  \Abs{ \left(\fy + \varepsilon\phiy\right)  + \frac{x}{2}} \right)  \left(  \Sgn\left(  \left(\fx + \varepsilon\phix\right) - \frac{y}{2}\right) \phix   \right)$$
and 
$$    -\Sgn\left( \Abs{ \left(\fx + \varepsilon\phix\right)  - \frac{y}{2}}  -  \Abs{ \left(\fy + \varepsilon\phiy\right)  + \frac{x}{2}} \right)  \left(  - \Sgn\left(  \left(\fy + \varepsilon\phiy\right)  + \frac{x}{2}\right) \phiy   \right)  $$

That gives us first variation of perimeter being$\frac{1}{2} \iint_{\Omega} \nabla \varphi \cdot v dx dy$, where  
$$ v = \left[
\begin{array}{c}
\left( 1 - \Sgn\left( \Abs{ \left(\fx + \varepsilon\phix\right)  + \frac{y}{2}}  -  \Abs{ \left(\fy + \varepsilon\phiy\right)  - \frac{x}{2}} \right) \right)  \left(  \Sgn\left(  \left(\fx + \varepsilon\phix\right) + \frac{y}{2}\right)  \right) \\
\left( 1 + \Sgn\left( \Abs{ \left(\fx + \varepsilon\phix\right)  + \frac{y}{2}}  -  \Abs{ \left(\fy + \varepsilon\phiy\right)  - \frac{x}{2}} \right) \right)  \left(  \Sgn\left(  \left(\fy + \varepsilon\phiy\right) - \frac{x}{2}\right)  \right)
\end{array} \right] $$

Evaluating at $\varepsilon = 0$ and applying integration by parts finishes the proof.
\end{proof}

A natural thing to wonder is where mean curvature is concentrated.  
The signum functions in these expressions suggest that mean curvature should be zero almost everywhere,
and we will see this is a fairly general phenomena.

\begin{theorem}\label{thm:suppcurves}
Say $f$ is a $C^2$ smooth function that parametrizes a regular surface $S = \{(x,y,f(x,y) \} \subset (H(\R), d_Q, \lambda)$ where $||\cdot||_Q = ||\cdot||_{\infty}$ so that $||\cdot||_{Q^*} = ||\cdot||_1$.
Then for either perimeter measure, the mean curvature distribution is supported on a set of $C^1$ curves along with isolated points for any definition of surface area.
\end{theorem}
\begin{proof}
Again we begin with $\Sa$ because it is the more simple one.
Looking at the divergence expression, away from $0$, the signum function is identically zero.
Thus to have $\M \neq 0$, we must look at where at least one $\Sgn$ has argument zero.
Equivalently, from the PDE expression, $\M$ is only non-zero when at least one $\delta$'s has argument zero, hence 
$$\supp(\M) \subset \left\{ (x,y,f(x,y)) \in S \mid \fx + \frac{y}{2} = 0  \right\} \bigcup \left\{ (x,y,f(x,y)) \in S \mid \fy - \frac{x}{2} \right\}.$$
Since these are level sets, they are closed, and their union is a closed set.

Now consider the set $A  = \left\{ (x,y,f(x,y)) \in S \mid \fx + \frac{y}{2} = 0  \right\}$.
At any point $p$ in the interior $A^\circ$, we can take integrals and derivatives of $f$ in a neighborhood of $p$,
and solve $f$ to be equal to $-\frac{xy}{2} + g(y) + c$, and hence $\fxx (p) = 0$.
We can do similar for the other set, and upgrade our previous statement to:
$$\supp(\M) \subset \partial \left( \left\{ (x,y,f(x,y)) \in S \mid \fx + \frac{y}{2} = 0  \right\} \bigcup \left\{ (x,y,f(x,y)) \in S \mid \fy - \frac{x}{2} = 0 \right\} \right).$$

To conclude the theorem as stated, we note that for a level set $F = const$ to be a $C^1$ curve, it has to have non-vanishing gradient.
In the case of $A$, the gradient is precisely $\left( \fxx, \fxy + \frac{1}{2} \right)$.
Inside the set $\partial A$, since we took $f$ to be $C^2$, we have that the set $\fxx \neq 0$ is composed of open sets (in the subspace topology) and isolated points.
Along those open sets of $\partial A$ then, $\partial A$ it is a $C^1$ curve and thus $\partial A$ is a union of $C^1$ curves and isolated points.
Note that the boundary of a union is contained in the union of the boundaries, and we have our theorem as desired for $\Sa$.

The argument for $\ASa$ is identical, except that in order to do it thoroughly, one needs to compute the PDE that comes out of the divergence,
which is not nearly as illuminating as it is wasteful to write out.
The important part of it is that in taking the derivative of the $1+\Sgn$ and $1-\Sgn$ pieces, we get new terms of the form
$\fxy + \frac{1}{2}$ and $\fyx - \frac{1}{2}$.
However, the same exact argument goes through, and we have that these are precisely the conditions for which the level set is a $C^1$ curve, and the result follows.
\end{proof}

We can then see that for many choices of $f$, the set $\supp(\M)$ is measure zero.
\begin{corollary}
Consider $f$ a $C^2$ function that parametrizes a surface in $(H(\R), d_Q, \lambda)$.
If the isolated points of $\left\{(x,y) \in \Omega \mid \fxx = 0 \text{ or } \fyy = 0  \right\}$ 
are a measure zero subset, then the mean curvature distribution is almost everywhere zero.
In particular, if $\left\{(x,y) \in \Omega \mid \fxx = 0 \text{ or } \fyy = 0  \right\}$ 
is contained in the set of characteristic points of a regular surface, then the mean curvature distributrion is almost everywhere zero.
\end{corollary}

One needs to be careful to note that since $\M$ is a distribution and not a function, $\M$ being almost everywhere zero does not imply that that the first variation of perimeter is zero.
This is precisely the same reasoning that integrating the $\delta$ distribution (which is almost everywhere. $0$) yields a value of $1$.

\begin{remark}
The more general case surfaces parametrized by $F : \R^3 \to \R^3$ is similar, but I will omit said painful computations for brevity.
Instead I can ask you to recall that vertical planes have Eucludean surface area times a factor that depends on the plane's orientation.
Since mean curvature of a plane in Euclidean space is zero, vertical planes have zero mean curvature in our contexts as well.
\end{remark}

\begin{example}[All planes are flat]
Consider a linear function $f(x,y) = k_1 x +  k_2 y$ with $k_1, k_2 \neq 0$ that parametrizes a plane in $(H(\R), d_Q, \lambda)$ 
where $||\cdot||_{Q} = ||\cdot||_\infty$. 
All the second derivatives are zero, so for Minkowski content, we have that mean curvature is zero everywhere.
On the other hand, this means that $\fxy \neq -1/2$ and $\fyx \neq 1/2$, so the terms that show up in anti-Minkowski content are a priori non-vanishing.
However, if you carry out the divergence, there will be a $+1/2$ contribution to the integral from the $\fxy + \frac{1}{2}$ term 
and $-1/2$ contribution from the $\fyx - \frac{1}{2}$ that cancel each other out. 
so the mean curvature for anti-Minkowski content is zero everywhere as well.
\end{example}

\begin{example}
Consider a quadratic function $f(x,y) = k_1 x^2 +  k_2 y^2$ with $k_1, k_2 \neq 0$.
This has $\fxx, \fyy \neq 0$ at all points..
$\fx - y/2 = 2 k_1 x  - y/2$ is zero exactly along the line $y = 4  k_1 x$, and similarly $\fy + x/2 = 0$ along the line $y = x /(4 k_2)$.
If we are thinking about mean curvature for Minkowski content, we have that it is zero everywhere except these lines.
\end{example}

Similar is true of we add in more terms to our quadratic function. 
From these two examples we have a nice fact:

\begin{example}
Since the metric sphere is the union of piecewise quadratic functions and vertical planes,
it has mean curvature zero almost everywhere for both perimeter measures.
Furthermore, the metric ball for any CC-metric arising from a polygonal norm 
has that same characterization, and thus also has mean curvature zero almost everywhere.
\end{example}

This behavior contrasts the behavior of Pansu's bubble in the sub-Riemannian case, which has constant mean curvature.
We can apply the corollary to see that this behavior is not unique to polygonal sub-Finsler spheres.

\begin{example}
The sub-Riemannian unit sphere and Pansu's bubble set have mean curvature zero almost everywhere for both perimeter measures.
\end{example}

We can now extend this to the general case of $Q$ a polygon to get  our second main theorem.

\newtheorem*{result2}{Theorem \ref{thm:meancurvatureismean}}
\begin{result2}[Mean curvature of smooth surfaces, polygonal case]
Say $F$ is a $C^2$ smooth function that parametrizes a surface $S$ in $(H(\R), d_Q)$ where $Q$ is a polygon.
Then for either definition of surface area, the mean curvature distribution of $F$ is supported on a union of $C^1$ curves and isolated points,
and hence is almost-everywhere $0$.
\end{result2}

\begin{proof}
The crux of the argument in the previous chapter was that we were taking divergence of signum functions, which came from distributional derivatives of absolute values. 
No matter which choice of surface area we make, it is the maximum over finitely many vertices of some centrally symmetric polygon.
Furthermore, since the figure $Q$ must be centrally symmetric, every vertex has it's opposite included,
and the max over opposite vertices is the absolute value.
This means that the integrand is the max of terms of the form $\Abs{a \fx + b\fy - \frac{ay - bx}{2} }$ for every $(a,b) \in v(Q)$.
Using the formula $\max\{x,y\} = \frac{x + y + | x - y |  }{2}$, we have that the integrand is a sum of absolute values in this manner.
Then taking the derivative with respect to $\varepsilon$, we have a combination of signum functions, 
which implies that when we take the divergence we will have a $\delta$ functions.

Just as before, the support of $\M$ is contained in the set where the argument of $\delta$ (and $\Sgn$) is $0$.
Consider the possibility that the equation $a \fx + b\fy - \frac{ay -bx}{2} = 0$ is satisfied.
We can rewrite this to be the PDE $\fx + \left( \frac{b}{a} \right) \fy = \frac{ay - bx}{2a}$.
Now we do a linear change of variables to reduce this to the $L^1$ case.
\[\begin{array}{ccc}
\alpha  = Ax + By & \implies & \fx = A \falpha + C \fbeta \\
\beta  = Cx + Dy & & \fy = B \falpha + D \fbeta
\end{array}\]

we then have
$$\fx + \left(\frac{b}{a}\right) \fy =  \left(A \falpha + C \fbeta\right) + \left(\frac{b}{a}\right)\left(B \falpha + D \fbeta\right)  $$
$$\implies \left(A + \frac{b}{a}B\right)\falpha + \left( C + \frac{b}{a}D \right)\fbeta  $$

To make this as nice as possible, let $A = 1, B = 0 , C = \frac{b}{a}, D = -1$, which obeys $AD - BC \neq 0$,
and makes $\alpha = x$ and $\beta = \frac{b}{a}x - y \implies y = \frac{b}{a} \alpha - \beta $.
Our PDE is then

$$ \falpha = \frac{ a(\frac{b}{a} \alpha - \beta) - b \alpha }{2a} = \frac{ b \alpha  - a \beta - b\alpha}{2a} = \frac{-\beta}{2} \implies \falpha + \frac{\beta}{2} = 0$$
We have now reduced this to exactly the same setup as the set $A$ in the proof of Theorem \ref{thm:suppcurves}, and the same argument goes through.
Since this works for generic $(a,b) \in v(Q)$, we can do this for every term in the integral.
\end{proof}

One important takeaway from this is that for a polygonal CC-metric, 
first variation of perimeter is of no help in helping solve the isoperimetric problem.
Ideally, one might hope to impose a fixed volume condition 
and utilize mean curvature flow to give numerical approximations to the isoperimetrix, but alas,
for functions where the support of $\M$ is measure zero this falls apart (which includes most examples!).
As said before, even if $\M$ has support a measure zero set, since it is a distribution 
it may still integrate to a non-zero value in the manner of a $\delta$ function.
When it does integrate to zero, such as on a plane, we can say that perimeter is locally constant under small perturbations.

\section{Q-Bubble Sets} \label{sec-bubble}
\subsection{Bubble Sets}
Recall that the Legendrian foliation of a surface records the intersection of horizontal planes and tangent planes.
Pansu showed that in the sub-Riemannian case, up to dilation/translation there is a unique surface that has Legendrian foliation by geodesics, and it is as follows:
\begin{definition}
The Pansu bubble set is the convex topological sphere in $H(\R)$ achieved by taking all sub-Riemannian geodesics from the origin to $(0,0,1)$. 
\end{definition}

\begin{figure}[h!]
\centering
\includegraphics[scale = 0.33]{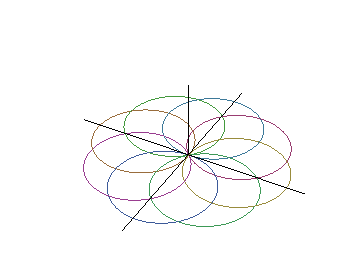}
\includegraphics[scale = 0.33]{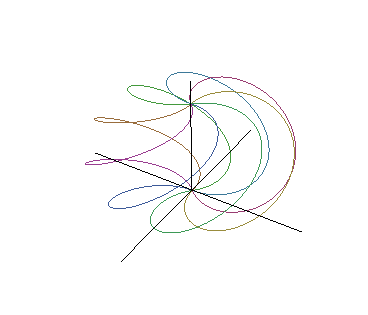}
\includegraphics[scale = 0.33]{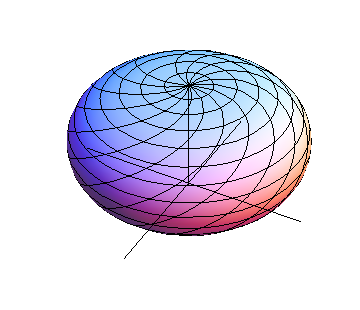}
\caption{ Construction of Pansu's bubble set}
\end{figure}

They are in fact, the only surfaces with constant mean curvature in the sub-Riemannian case \cite{RR},
and the subject of Pansu's conjecture Pansu conjectured a solution to the isoperimetric problem (Conjecture \ref{conj:pansu}).
More generally, I give the following similar definition. 
\begin{definition}
Define the $Q$-bubble set $\Bubb(Q)$ as the set of all $d_Q$ geodesics from the origin to $(0,0,1)$.
\end{definition}

\begin{proposition}\label{prop:sphere}
For all $Q$, $\Bubb(Q)$ is a topological sphere.
\end{proposition}
\begin{proof}
Recall from path lifting that geodesics from the origin to $(0,0,1)$ are lifts of closed curves in the plane that enclose area $1$ and with minimal Minkowski length 
in the planar metric $||\cdot||_Q$.
This means that such planar paths must follow the boundary of an isoperimetrix for $|| \cdot ||_Q$, 
which we know from Busemann (Theorem \ref{thm:busemanstheorem}) are all dilates of $I = e^{\pi /2} Q^*$, a topological circle.
Thus the set of geodesics between the origin and $(0,0,1)$ can be parametrized by the continuous family $\gamma_{s}(t)$, where $t \in [0,1]$ and $s \in \partial I$ is a choice of start point,
$\gamma_s(0)$ is the origin for all $s$, and  $\gamma_s (1)$ is $(0,0,1)$ for all $s$.

Now consider the possibility that $\gamma_s(t) = \gamma_{s'}(t')$ .  
For this to occur, we would need to have two chords of $I$ that enclose the same area.
From convexity, there are two possibilities:  either there is a one parameter family of such chords or there are exactly two.
In the case of the one parameter family, the area must be strictly increasing as you vary the parameter, else the plane figure is not convex, which is a contradiction.
In the case of two chords, one path around $\partial I$ must contain the other, and hence has higher area, which is a contradiction.

Therefore $\Bubb(Q)$ is the image of a sphere under an embedding, and thus is a topological sphere in $H(\R)$.
\end{proof}

We can use this parametrization of bubble sets to construct them explicitly.
Visually we can think of them as follows:  take all translates of the planar figure $I$ and lift them to achieve a surface.
For example, if we have $Q$ chosen so that our norm is $||\cdot||_1$, $I$ is the unit square, 
so $\Bubb(Q)$ is the set of all lifts of squares through the origin.
The construction generalizes Pansu's bubble set construction, and we have the following generalization of a theorem of Pansu's.
\begin{proposition}
The Legendrian foliation of $\Bubb(Q)$ is by $d_Q$ geodesics.
\end{proposition}
\begin{proof}
Recall that the tanget space of a surface at a point is the space spanned by all tangent vectors of all curves on the surface at that point.
The foliation of a regular surface is comprised of curves that have tangent vector in the horizontal plane field (except on a measure zero set of points).
A $Q$ bubble set is the union of horizontal curves (tangent vectors are in the horizontal plane field) so these curves must be the foliation.
By construction, these curves were taken to be geodesics.
\end{proof}

\subsection{Polygonal bubble sets}

In many ways constructing a bubble set is a dual notion to that of constructing a metric sphere.
To construct a sphere, you can start with the unit ball $Q$ of your norm $||\cdot||_Q$ in the plane,
and for each point in $Q$ seek out a geodesic of length 1 that encloses the most area and ends at said point.
The endpoint of the lifted path is now a point in the sphere.
In order to find such a geodesic in the interior of $Q$, the path must travel along the boundary of the largest possible dilate of $I$, and move $1$ unit total.
By contrast, to construct a $Q$-bubble set, you can start with isoperimetrix $I$ associated to $||\cdot||_Q$ in the plane.
A path with endpoint in the interior of $I$ must travel along the boundary of a translate of $I$ and move for as much arclength as possible.
In the case where $Q$ is a polygon, the metric balls and bubble sets have much structural similarity.

\begin{theorem} [Polygonal Bubble Set Structure Theorem]
Let $Q$ be a $2n$ sided, centrally symmetric polygon, 
let $I$ be a planar figue with $\partial I$ an isoperimetrix for $||\cdot||_Q$ enclosing unit area,
and $S$ the translate of $\Bubb(Q)$ by $(0,0,-1/2)$.
Then $S$ can be described as  $S = \{ (x,y,z) : |z| \leq f(x,y),   (x,y) \in 2I \},$
where $f : \R^2 \to \R $ is a piece-wise quadratic function with $(2n)(n-1)$ quadrilateral subdomains
on which it is exactly a polynomial. 
Furthermore, $S \cap (\partial (2I)\times \R)$ is comprised of $2n$  vertical quadrilateral walls. 
\end{theorem}

Unsurprisingly, the proof is similar to the proof of the structure theorem for CC-spheres found in \cite{M&M}.

\begin{proof}
We appeal to the parametrization in the proof of Proposition \ref{prop:sphere},
which gives that $\Bubb(Q)$ is parametrized by the family of geodesics $\gamma_s(t)$ where $s$ is a choice of start point along $\partial I$.
We first note that $Q$ being a centrally symmetric polygon with $2n$ sides implies $I$ is as well,
and that taking the union of translates of $I$ passing through the origin is $2I$, which is the support of $f$.

A path in the parametrization starts and ends on sides of $I$, so there are $(2n)(2n) = 4n^2$ combinatorial types.
Clearly if the path starts and ends on the same side, it either does not enclose any area or has enclosed all of the area already.
In either case, it is locally constant and thus choosing points on the same edge parametrizes a line on the surface (and not a region).
Similarly, if the path ends on the edge parallel to the one it started on, then as you change start and end points, the $x,y$ coordinates move along a line in the plane 
and the area changes linearly (see the picture on the right in Figure \ref{fig:chords}).  This then parametrizes the image of a rectangle under a linear map, and thus is vertical, i.e. parallel to the $xy$ plane.

So we must consider how much area we enclose if the path ends on any of the other $(2n-2)$ sides.
To do so, consider a trick that is said to be due to Gauss:  the shoelace polynomial (it can be proven from Green's theorem).
Say you have a polygon with vertices $(x_i, y_i)$. 
Then the area of the polygon is 
$$A = \frac{1}{2} \Abs{ 
\sum_{i=1}^{n} \det
\left[\begin{array}{cc}
x_i & x_{i+1}\\
y_i & y_{i+1}
\end{array} \right]
  } .$$

A point $(x,y)$ is associated to such a polygon whose vertices come from a translate of $I$ along with the origin and $(x,y)$.
We can recenter this by taking the plane figure $I$ and translating by $-s$  as in Figure \ref{fig:translating}.
To prove that the area enclosed (the $z$ coordinate) is a quadratic function, we need only prove that the coordinates of the points $(x_i, y_i)$ are linear in $x$ and $y$ and apply the shoelace formula.
{\begin{figure}[h!]
  \centering
\begin{tikzpicture}[scale=1]
\filldraw[fill=pink]  (-0.75,-1) -- (0,-1) --(1,0)  -- (1,1-.75);
\draw[<->] (-1.5,0) -- (1.5,0);
  \draw[<->] (0,-1.5) -- (0,1.5);
\draw[blue] (1,0)  -- (1,1) -- (0,1) -- (-1,0) -- (-1,-1) -- (0, -1) -- (1,0) -- (1,1);
\draw[ ->] (-0.75,-1)  -- (1,1-.75);
\node at (-0.75,-1.15) {$s$};

\node at (2,0) {$\Rightarrow$};
\end{tikzpicture}
\begin{tikzpicture}[scale=1]
\draw[<->] (-0.25,0) -- (2,0);
  \draw[<->] (0,-0.5) -- (0,2);
\filldraw[fill=pink]  (0,0) -- (1-.25, 0) -- (2-.25,1) -- (2-.25,2-.75);
\draw[blue] (2-.25,1)  -- (2-.25,2) -- (1-.25,2) -- (0-.25,1) -- (0-.25,0) -- (1-.25, 0) -- (2-.25,1) -- (2-.25,2);
\draw[ ->] (0,0)  -- (2-.25,2-.75);
\node at (2.25,2-.65) {$(x,y)$};
\node at (0.75,-.15) {$(x_2, y_2)$};
\node at (2.35,1) {$(x_3, y_3)$};
\end{tikzpicture}
\caption{Recentering a path along $\partial I$.}\label{fig:translating}
\end{figure}
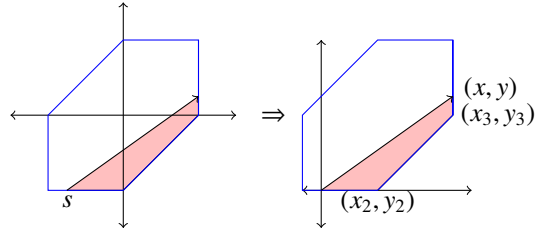}

Note that for a fixed $(x,y)$, one can solve for $(x_i, y_i)$ using plane geometry and observe they can be expressed as a linear combination of $x$ and $y$.
For example, in the figure $(x_2, y_2)$ is equal to $(x-1, 0)$.  Varying $(x,y)$ is equivalent to varying the chord along the edges they lay.
Furthermore, changing the endpoint does not change the values of the $(x_i, y_i)$ upon recentering -- only changing the start point does.
Locally, the start point $s$ can only change linearly along the edge it lays, which is linear, and hence since we recenter by $-s$, 
the $(x_i, y_i)$ vary linearly as well, and hence area varies quadratically.

Now we reduce by symmetry: picking any of the $n-1$ edges to end on before the parallel edge to the one you started makes you negative since we left-translated,
and since $Q$ is is centrally symmetric, so is $I$, and we are done.
\end{proof}

\begin{example}[Construction of the square bubble set]
Consider $Q$ to be the diamond defining the norm $||\cdot||_1$.  Then following the theorem, we have that if
we take $I$ is the square with sidelengths $1$, we know $x, y \in [-1,1] = 2I$.
Say $x, y \geq 0$, since we are traversing the square, the only way to end up in the first quadrant is to traverse from the western edge to the northern edge counter-clockwise.

Calculating area enclosed is actually easier to do by the complement.
The area enclosed is $Area(I)$ minus the area enclosed by the polygon with edges $(0,0), (0,y), (x,y)$, and thus the maximum $z$ coordinate is is $1- xy/2$.
Doing this across all four quadrant yields $1 - \Abs{xy}/2$, and translating by $(0,0,-1/2)$ for convenience yields a set that is set that is symmetric about the $xy$-plane
that can be described as 
$$ (0,0,-1/2)*Bubb(Q) =  \left\{ (x,y,z) \, \left| \,  |z| \leq \frac{1 - \Abs{xy}}{2} \text{ and } x,y \in [-1,1] \right\} \right.,$$
\end{example}

\begin{figure}[h!]
\centering
\includegraphics[scale = 0.33]{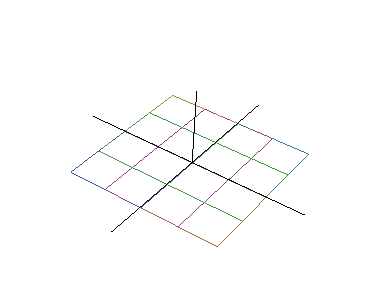}
\includegraphics[scale = 0.33]{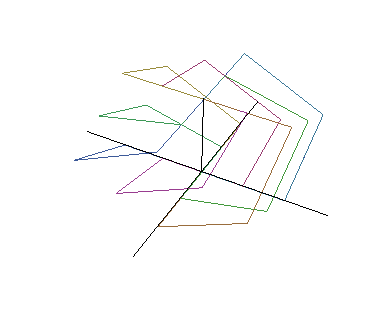}
\includegraphics[scale = 0.25]{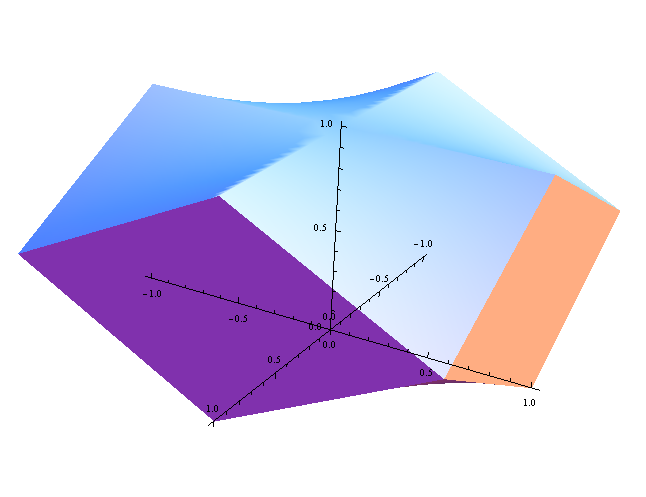}
\caption{ The square bubble set is constructed by taking all $d_Q$ geodesics from the origin to $(0,0,1)$ for the CC-metric arising from $||\cdot||_1$,
which are lifts of squares.}\label{figure:l1bubb}
\end{figure}

\begin{remark}
The square bubble set is another surface that has 
mean curvature zero almost everywhere for both perimeter measures in the case of $||\cdot||_Q = ||\cdot||_1$ or $||\cdot||_{\infty}$
due to being made up of the graph of a piecewise quadratic function and vertical walls.
In fact, all polygonal bubble sets have this behavior in all polygonal CC-metrics.
\end{remark}

\subsection{Best constants for the cases arising from $||\cdot||_1$ and $||\cdot||_{\infty}$}\label{subsec:bestconst}
To conclude this work, I will give the best known constants for the sub-Finsler cases that arise from the two simplest $p$-norms. 
Noting that Minkowski content in one is anti-Minkowski content in the other, for the remainder of the section 
I fix $Q$ to be such that $||\cdot||_1$.

From the bounds from Section \ref{subsec:bounds}, we know that the sub-Riemannian isoperimetric ratio for the Pansu bubble set ($0.321519$) is a lower bound on 
the Isoperimetric constant for Minkowski content.  Furthermore, if we are in a setting under which Pansu's conjecture holds, we get an upper bound ($0.454697$).
Similarly, we have a lower bound for anti-Minkowski content (again $0.321519$) and, when Pansu's conjecture holds, an upper bound ($0.643038$). 
As such any purported isoperimetrix should have isoperimetric ratios in those ranges.

Using the integral formulas from Lemma \ref{lem:simple}, one can compute these using a computer and observe the following:

\begin{example}Computation for the metric ball $B_{d_Q}$ yields an isoperimetric ratio of $0.308626$ for Minkowski content and $0.154422$ for anti-Minkowski content.
\end{example}

\begin{example}Computation for the $Q$-bubble set $\Bubb(Q)$ yields an isoperimetric ratio of $0.284938$ for Minkowski content and $0.379918$ for anti-Minkowski content.
\end{example}

\begin{example}Computation for the $Q^*$-bubble set $\Bubb(Q^*)$ yields an isoperimetric ratio of $0.268642$ for Minkowski content and $0.228175$ for anti-Minkowski content.
\end{example}

\begin{example}Computation for Pansu's bubble set $\Bubb(\D)$ yields an isoperimetric ratio of $0.357117$ for Minkowski content and $0.50504$ for anti-Minkowski content.
\end{example}

Surprisingly, we have:

\newtheorem*{result3}{Proposition \ref{thm:pansuisgood}}
\begin{result3}[Pansu bubble sets appear optimal]
Let $(H(\R), d_Q, \lambda)$ be the sub-Finsler metric measure space given by
the Heisenberg group $H(\R)$  equipped with $d_Q$ the CC-metric arising from the norm $||\cdot||_1$  ($Q$ the unit diamond) and  either notion of surface area. 
Then the isoperimetric ratio of the Pansu bubble set is higher than that of the the diamond bubble set, the square bubble set, and the metric balls for $d_Q$ and $d_{Q^*}$.
\end{result3}

I have computed isoperimetric ratios for various other surfaces (ex. polyhedra, Euclidean sphere), and observed that none of these have a larger ratio than the Pansu bubble set.

As mentioned in the background section, in the Finsler normed planes circles (the Riemannian isoperimetrix) are not isoperimetrically optimal for normed planes given by a plane figure $Q$, but rather the boundary of $e^{\pi/2}Q^*$ is.
This makes it very fascinating that in the sub-Finsler Heisenberg group, the surface given by lifts of circles (the Pansu bubble set, the conjectured sub-Riemannian isoperimetrix) appears to be a  sub-Finsler isoperimetrix and while neither lifts of the plane figure ($\Bubb(Q^*)$) nor lifts of its dual ($\Bubb(Q)$) are.
With that in mind, I make the following conjecture regarding isoperimetrices of the sub-Finsler Heisenberg groups in relation to the sub-Riemannian case.

\newtheorem*{result4}{Conjecture \ref{conj:genpansu}}
\begin{result4}[Generalized Pansu Conjecture]
For $(H(\R), d_Q, \lambda)$ with any CC-metric $d_Q$ and either notion of surface area, 
the Pansu bubble set is the unique isoperimetrix (up to dilation and translation).
\end{result4} 



\begin{thebibliography}{CDPT07}
\bibitem[APF98]{JCA}
Juan~Carlos \'Alvarez-Paiva and Carlus~Dur\'an Fern\'andez.
\newblock {\em An introduction to Finsler geometry}.
\newblock Notas de la Escuela Venezolana de Mat\'ematicas, 1998.

\bibitem[APT04]{JCA2}
Juan~Carlos \'Alvarez-Paiva and Anthony Thompson.
\newblock {\em Volumes on Normed and Finsler Spaces}.
\newblock MSRI Publications, 2004.

\bibitem[Bus47]{BusemannIso}
Herbert Busemann.
\newblock { The Isoperimetric Problem in the Minkowski Plane}.
\newblock {\em American Journal of Mathematics}, 69(4):863--871, 1947.

\bibitem[Bus55]{BusemannGeod}
Herbert Busemann.
\newblock {\em The Geometry of Geodesics}.
\newblock Academic Press Inc., New York, N.Y., 1955.

\bibitem[CDPT07]{Capogna}
Luca Capogna, Donatella Danielli, Scott~D. Pauls, and Jeremy Tyson.
\newblock {\em An Introduction to the Heisenberg Group and the Sub-Riemannian
  Isoperimetric Problem: 259 (Progress in Mathematics)}.
\newblock Birkh\"user, 2007.

\bibitem[DM14]{M&M}
Moon Duchin and Christopher Mooney.
\newblock {Fine asymptotic geometry of the Heisenberg group}.
\newblock {\em Indiana University Math Journal}, 63(3):885--916, 2014.

\bibitem[DS14]{M&Mike}
Moon Duchin and Michael Shapiro.
\newblock { Rational growth in the Heisenberg group}.
\newblock {\em http://arxiv.org/abs/1411.4201}, 2014.

\bibitem[Mon00]{Monti3}
Roberto Monti.
\newblock Some properties of carnot-carath\'eodory balls in the heisenberg
  group.
\newblock {\em Atti della Accademia Nazionale dei Lincei. Classe di Scienze
  Fisiche, Matematiche e Naturali. Rendiconti Lincei. Matematica e
  Applicazioni}, 11(3):155 -- 167, 2000.

\bibitem[Mon08]{Monti1}
Roberto Monti.
\newblock Heisenberg isoperimetric problem. the axial case.
\newblock {\em Advances in Calculus of Variations}, 1(1):93 -- 121, 2008.

\bibitem[MR09]{MR}
Roberto Monti and Matthieu Rickly.
\newblock Convex isoperimetric sets in the heisenberg group.
\newblock {\em Ann. Sc. Norm. Super. Pisa Cl. Sci.}, 8(2):391 -- 415, 2009.

\bibitem[Pan82]{Pansu2}
Pierre Pansu.
\newblock Une in\'egalit\'e isop\'erim\'etrique sur le groupe de heisenberg.
\newblock {\em Comptes rendus de l'Acad\'emie des sciences paris s\'er I},
  295(2):127 -- 130, 1982.

\bibitem[Pan83]{Pansu3}
Pierre Pansu.
\newblock Croissance des boules et des g\'eod\'esiques ferm\'ees dans les
  nilvari\'et\'es.
\newblock {\em Ergodic Theory Dynami. System}, 3(2):415--445, 1983.

\bibitem[Pan84]{Pansu1}
Pierre Pansu.
\newblock An isoperimetric inequality on the heisenberg group.
\newblock In {\em Rendiconti del Seminario Matematico Universit\`a e
  Politecnico di Torino}, pages 159--174, 1984.
\newblock Conference on differential geometry on homogeneous spaces (Turin
  1983).

\bibitem[RR06]{RR}
Manuel Ritor{\'e} and C{\'e}sar Rosales.
\newblock Rotationally invariant hypersurfaces with constant mean curvature in
  the heisenberg group $\mathbb{H}^n$.
\newblock {\em The Journal of Geometric Analysis}, 16(4):703, 2006.

\bibitem[RR08]{RR2}
Manuel Ritor{\'e} and C{\'e}sar Rosales.
\newblock Area-stationary surfaces in the heisenberg group $\mathbb{H}^1$.
\newblock {\em Adv. Math}, 219(2):633--671, 2008.

\bibitem[{S}{\'a}n17]{thesis}
Ayla~P. {S}{\'a}nchez.
\newblock {\em A Theory of Sub-Finsler Surface Area in the Heisenberg Group}.
\newblock Tufts University, ProQuest Dissertations Publishing, 2017.

\bibitem[Ver09]{Vershynin}
Roman Vershynin.
\newblock { Lectures in Geometric Functional Analysis}.
\newblock {\em
  https://www.math.uci.edu/~rvershyn/papers/GFA-book.pdf}, 2009.
\end{thebibliography}


\end{document}